\documentclass{amsart}
\usepackage[english]{babel}
\usepackage{mathrsfs}
\usepackage{amsmath}
\usepackage{amsthm}
\usepackage{amssymb}
\usepackage{indentfirst}
\usepackage[all]{xy}
\usepackage{tikz}
\usepackage{tikz-cd}
\usetikzlibrary{shapes.geometric, calc,matrix,arrows,decorations.pathmorphing,arrows.meta}
\usepackage{graphicx}
\usepackage{graphics}
\usepackage{caption}
\usepackage{extarrows}
\usepackage[enableskew, vcentermath]{youngtab}
\usepackage[utf8x]{inputenc}
\usepackage[colorlinks]{hyperref}

\makeatletter
\newcommand{\setword}[2]{%
  \phantomsection
  #1\def\@currentlabel{\unexpanded{#1}}\label{#2}%
}
\makeatother

\theoremstyle{plain}
\newtheorem{thm}{Theorem}[section]

\newtheorem{lem}[thm]{Lemma}
\newtheorem{cor}[thm]{Corollary}
\newtheorem{prop}[thm]{Proposition}

\theoremstyle{definition}
\newtheorem{defn}[thm]{Definition}

\newtheorem{rem}[thm]{Remark}
\newtheorem{question}[thm]{Question}

\newtheorem{strategy}[thm]{Strategy}
\newtheorem{notation}[thm]{Notation}

\renewcommand{\Im}{\operatorname{Im}}

\newcommand{\Gam}{i_*\mathcal O_{\gamma\times\gamma}(7(p_0\times\gamma)+\Delta)}

\newcommand{\CC}{\mathbb{C}}
\newcommand{\PP}{\mathbb{P}}
\newcommand{\QQ}{\mathbb{Q}}

\newcommand{\ZZ}{\mathbb{Z}}

\newcommand{\cE}{\mathcal{E}}
\newcommand{\cF}{\mathcal{F}}
\newcommand{\cG}{\mathcal{G}}
\newcommand{\cH}{\mathcal{H}}
\newcommand{\cI}{\mathcal{I}}
\newcommand{\cJ}{\mathcal{J}}

\newcommand{\cL}{\mathcal{L}}

\newcommand{\cO}{\mathcal{O}}

\newcommand{\cT}{\mathcal{T}}

\newcommand{\cX}{\mathcal{X}}

\newcommand{\fM}{\mathfrak{M}}

\DeclareMathOperator{\ch}{ch}
\DeclareMathOperator{\CH}{CH}
\DeclareMathOperator{\Coh}{Coh}

\newcommand{\divv}{\operatorname{div}}

\newcommand{\Ext}{\operatorname{Ext}}

\newcommand{\Gr}{\operatorname{Gr}}

\DeclareMathOperator{\K}{K}

\DeclareMathOperator{\M}{M}
\DeclareMathOperator{\Mon}{Mon}

\newcommand{\NS}{\operatorname{NS}}
\DeclareMathOperator{\OG}{OG}
\newcommand{\Or}{\operatorname{O}}

\newcommand{\Ort}{\operatorname{O}}

\DeclareMathOperator{\Pic}{Pic}

\newcommand{\rk}{\operatorname{rk}}

\newcommand{\Sym}{\operatorname{Sym}}

\DeclareMathOperator{\Supp}{Supp}
\DeclareMathOperator{\td}{td}

\begin{document}
\title{Rational Curves on O'Grady's tenfolds}
\author{Valeria Bertini}
\address{Fakult\"at f\"ur Mathematik - TU-Chemnitz, Deutschland}
\email{valeria.bertini@math.tu-chemnitz.de}

\begin{abstract}
We study the existence of ample uniruled divisors on irreducible holomorphic symplectic manifolds that are deformation of the ten dimensional example introduced by O'Grady in \cite{OG10}. In particular, we show that for any polarized \(\OG10\) manifold lying in four specific connected components of the moduli space of polarized \(\OG10\) manifolds there exists a multiple of the polarization that is the class of a uniruled divisor.  
\end{abstract}
\keywords{Irreducible holomorphic symplectic variety; O'Grady's ten dimensional example; rational curves; ample uniruled divisors.  \\ MSC 2010 classification: 14C99; 14D99; 14J40.}

\maketitle

\tableofcontents

\section{Introduction and notation}\label{section:introduction}

\subsection{Rational curves on IHS manifolds and main result}

A compact complex K\"ahler manifold $X$ is called irreducible holomorphic symplectic (IHS) if it is simply connected and $H^0(X,\Omega^2_X)=\CC\cdot\sigma$, with \(\sigma\) an everywhere non-degenerate two form. IHS manifolds have been introduced by Beauville in \cite{beauville.varietes}. As straightforward consequence of the definition, an ihs manifold has even complex dimension and trivial canonical bundle. In dimension two, IHS manifolds are the $\K3$ surfaces. In higher dimension, examples of IHS manifolds are given by: the Hilbert scheme \(\K3^{[n]}\) of length \(n\) 0-dimensional subschemes on a \(\K3\) surface and the generalized Kummer variety \(\K_n(A)\) on an abelian surface \(A\), both defined for any \(n\ge 2\) and introduced by Beauville in \cite{beauville.varietes}; the six and ten dimensional examples introduced by O'Grady in \cite{OG6} and \cite{OG10} respectively. Up to deformation, no other example of IHS manifold is nowadays know. An IHS manifold will be called of \(\text{K}3^{[n]}\), \(\K_n(A)\), \(\OG6\) or \(\OG10\) type if it is deformation equivalent to \(\K3^{[n]}\), \(\K_n(A)\), the O'Grady's six dimensional example or the O'Grady's ten dimensional example respectively.  

Given an IHS manifold \(X\), the second cohomology group \(H^2(X,\ZZ)\) has a lattice structure given by the Beauville-Bogomolov-Fuijiki quadratic form, that we will denote \(q_X\). The Beauville-Bogomolov-Fujiki form is a deformation invariant of the IHS manifold; it has been computed in the \(\K3^{[n]}\) and \(\K_n(A)\) case by Beauville in \cite{beauville.varietes}, and in the \(\OG6\) and \(\OG10\) case by Rapagnetta in \cite{rapagnettaOG6} and \cite{rapagnettaOG10}.
\\

On a projective \(\K3\) surface, any ample curve has a multiple that is linearly equivalent to the sum of rational curves, thanks to a result by Bogomolov and Mumford (see \cite{mori.mukai.uniruledness}). In \cite{beauville.voisin}, Beauville and Voisin pointed out a consequence of this result on the Chow-0 group of a projective \(\K3\) surface: any point on any rational curve on a projective \(\K3\) surface \(S\) gives the same 0-cycle \(c_S\), and the image of the intersection product \(\Pic(S)\otimes\Pic(S)\rightarrow \CH_0(S)\) is contained in \(\ZZ\cdot c_S\). 

The aim of this paper is to investigate the existence of rational curves on projective IHS manifolds of \(\OG10\) type.  We denote by \(\fM_{\OG10}^{pol}\) the moduli space of polarized IHS manifolds of \(\OG10\) type; an element in \(\fM_{\OG10}^{pol}\) is a pair \((X,c_1(H))\), where \(X\) is an \(\OG10\) type IHS manifold and \(c_1(H)\) is a polarization on it, i.e.\ it is the first Chern class of an ample divisor \(H\) on \(X\). The main result of this work is that for each \((X,c_1(H))\) lying in four specific connected components of \(\fM_{\OG10}^{pol}\), there exists a multiple of \(H\) that is linearly equivalent to the sum of uniruled divisors. More precisely, we can state this result as follows. Connected components of \(\mathfrak M^{pol}_{\OG10}\) are characterized by degree and divisibility of the polarization \(c_1(H)\) in the Beauville-Bogomolov-Fujiki lattice. This is a consequence of the study of the monodromy group \(\Mon^2(X)\) of an IHS manifold \(X\) of \(\OG10\) type, which has been recently computed by Onorati in \cite{onorati.monodromy} and coincides with the group \(\Ort^+(H^2(\OG10,\ZZ))\) of orientation preserving isometries; we refer to subsection \ref{subsection:monodromy.invariants} for more details about that. We denote by \(\mathfrak M_{(d,l)}\) the irreducible component of \(\mathfrak M_{\OG10}^{pol}\) given by degree \(d\) and divisibility \(l\). 

\begin{thm}[see Theorem \ref{thm.uniruled.conn.comp}] For any polarized irreducible holomorphic symplectic manifold $(X,c_1(H))\in\mathfrak M_{(10,1)}\cup\mathfrak M_{(28,1)}\cup\mathfrak M_{(178,2)}\cup \mathfrak M_{(448,2)}\subseteq\mathfrak M_{\OG10}^{pol}$ there exists a positive integer \(m\) such that the linear system \(|mH|\) contains an element that is uniruled. 
\end{thm}

As in the case of projective $\K3$ surfaces, the existence of rational curves ruling a divisor has a strong consequence on the Chow-0 group of the manifold. Namely, by the results of Charles, Mongardi and Pacienza in \cite{CPM} we deduce the following.

\begin{cor}\label{cor Chow} Let \(X\) be an \(\OG10\) manifold such that there exists a polarization \(h\) on \(X\) with \((X,h)\in \mathfrak M_{(10,1)}\cup\fM_{(28,1)}\cup\mathfrak M_{(178,2)}\cup \mathfrak M_{(448,2)}\). Then, given an irreducible uniruled divisor $i:D\hookrightarrow X$, the subgroup 
\[
\Im(i_*:\CH_0(D)\rightarrow\CH_0(X))_\QQ=:S_1\CH_0(X)_\QQ
\]
 is independent of the chosen irreducible uniruled divisor. Here we are using the notation \(\CH_0(X)_\QQ:=\CH_0(X)\otimes_\ZZ \QQ\).
\end{cor} 

\begin{rem} Observe that any \(\OG10\) manifold \(X\) admitting a Lagrangian fibration with a relative theta divisor satisfies the hypotheses of Corollary \ref{cor Chow}. Indeed, the Neron-Severi group of such variety contains an hyperbolic lattice: it is generated by the class \(b\) of the Lagrangian fibration, which is isotropic, and the relative theta divisor \(\theta\), which is such that \(q_X(\theta,b)=1\) (see \cite[Lemma 1]{sawon} and \cite[Lemma 3.5]{sacca.birational}).
\end{rem}

It is interesting to remark that the subgroup $S_1\CH_0(X)_\QQ$ introduced in Corollary \ref{cor Chow} is the first block of a conjectural filtration $S_\bullet \CH_0(X)_\QQ$ of the Chow-0 group \(\CH_0(X)_\QQ\) of a projective IHS manifold \(X\), introduced by Voisin in \cite[Definition 0.2]{voisin.remarks}. This conjectural filtration gives, in the case of projective IHS manifolds, a geometric realization of the well known Bloch-Beilinson filtration of $\CH_0(X)_\QQ$; we refer to \cite{voisin.remarks} for the details and the precise definitions.
\\

The existence of ample uniruled divisors on IHS manifolds of Beauville's deformation type has been studied by Charles, Mongardi and Pacienza \cite{CPM} in the $\K3^{[n]}$ case, and by Mongardi and Pacienza \cite{mongardi.pacienza}, \cite{mongardi2019corrigendum} in the \(\K_n(A)\) case. In both cases, the authors have proved that for any polarized IHS manifold out of finitely many connected components of the moduli space of polarized IHS manifolds of the respective Beauville's deformation type, there exists a multiple of the ample divisor that is linearly equivalent to a sum of uniruled divisors. In the $\K3^{[n]}$-case the existence result is optimal as a combination of \cite{CPM} and \cite{oberdieck.shen.yin}. 
\\

The strategy used in this article to prove the existence of ample uniruled divisors on \(\OG10\) manifolds is the following. We find explicit examples of polarized IHS manifolds \((X,h)\in \fM_{\OG10}^{pol}\) such that a multiple of the polarization \(h\) has a representative that is uniruled, and then conclude that the same holds true for any element in the connected component of \(\fM_{\OG10}^{pol}\) containing \((X,h)\), thanks to a result of deformation of rational curves ruling a divisor on an IHS manifold. This strategy was already used in \cite{CPM} and \cite{mongardi.pacienza} for the \(\K3^{[n]}\) and \(\K_n(A)\) case.

\subsection{Plan of the paper} The paper is organized as follows. In Section \ref{section:moduli} we recall the construction of the model of \(\OG10\) manifold we will work on, together with some relevant results that we will use in the following sections. In Section \ref{section:uniruled divisors} we define four uniruled divisors, and we present a strategy to check their positivity with respect to the Beauville-Bogomolov-Fujiki form; the positivity of the divisors will give the existence of ample uniruled divisors on a deformation of the IHS manifold we started with, as observed in Remark \ref{rem.need.positivity}. The strategy presented in Section \ref{section:uniruled divisors} is developed in Section \ref{section:generators Mukai and intersection} and Section \ref{section:class of the divisors}, where we conclude with Theorem \ref{thm.q(D1*)}, Corollary \ref{cor.D1*.ruled.pos} and Theorem \ref{thm.q(tildeD2)} that the four divisors introduced are indeed positive. In Section \ref{section:ample.uniruled.divisors} we finally conclude with Theorem \ref{thm.uniruled.conn.comp} the existence of ample uniruled divisors for four connected components of \(\fM_{\OG10}^{pol}\), as presented in the introduction; this is done thanks to a result of deformation of rational curves ruling a divisors on an IHS manifold, see Corollary \ref{cor.defo.curves}. We conclude the article presenting some possible natural development of our work.

\subsection{Notation} We will always work over the field of complex numbers $\CC$. Given an IHS manifold \(X\) of \(\OG10\) type, we will denote by \(q_{10}\) the associated Beauville-Bogomolov-Fujiki form. For an IHS manifold \(X\), the first Chern class homomorphism gives an inclusion of the Picard group \(\Pic(X)\) in the second cohomology group \(H^2(X,\ZZ)\); in other words, the Picard group is isomorphic to the Neron-Severi group \(\NS(X)\). For this reason, given a divisor \(D\in\Pic(X)\), by abuse of notation  we will still denote by \(D\) its class in \(\NS(X)\). Also, for any divisor \(D\) on \(X\) we will write \(q_{10}(D)\) for \(q_{10}(c_1(D))\).

\subsection*{Acknowledgements}
Most of the work contained in this article was carried out during my PhD at the Università degli Studi di Roma Tor Vergata and at the Université de Strasbourg, under the supervision of Gianluca Pacienza and Antonio Rapagnetta. It is my pleasure to thank them for their invaluable help and their support during all the stages of the work. The author strongly benefitted from the help and many discussions with Christian Lehn and Claudio Onorati. Furthermore, I wish to thank Samuel Boissièr and Arvid Perego for having carefully read my PhD thesis, and for their comments and suggestions. In particular, I thank Arvid Perego for having pointed out a blunder in the computations in the proof of Proposition \ref{prop.inters.D2}. Finally, I want to thank the anonymous referee for their comments and questions, which have improved the work. 
This work has been partially funded by the DFG through the research grant Le 3093/2-2, and by Portuguese national funds through FCT — Fundação para a Ciência e a Tecnologia — under the project EXPL/MAT-PUR/1162/2021 and through CMUP (Centro de Matematica da Universidade do Porto) under the
project UIDB/00144/2020.


\section{Recall on the ten dimensional O'Grady's example}
\label{section:moduli}

\subsection{Moduli spaces of sheaves}\label{subsection:moduli.sheaves} In this paper we will construct uniruled divisors on IHS manifolds of $\OG10$-type that are desingularization of moduli spaces of semistable sheaves on a projective $\K3$ surface. We summarize here the construction of such manifolds and we collect some results that we are going to use in what follows. 

Let $S$ be a projective $\K3$ surface. Given a coherent sheaf $E\in\Coh(S)$, its Mukai vector \(\mathfrak v(E)\) is defined as:
\[
\mathfrak v(E):=\ch(E)\sqrt{\td(S)}=(\rk(E),c_1(E),\ch_2(E)+\rk(E))\in H^*(S,\ZZ).
\]
Notice that the vector \(\mathfrak v(E)\) is written as a three-entries vector because \(H^*(S,\ZZ)\cong H^0(S,\ZZ)\oplus H^2(S,\ZZ)\oplus H^4(S,\ZZ)\), since \(S\) is a \(\K3\) surface. 
Given any element $v=(r,l,s)\in H^*(S,\ZZ)$, we say that $v$ is a Mukai vector if $r\ge0$ and $l\in \NS(S)$, and if $r=0$ then either $l$ is the first Chern class of an effective divisor, or $l=0$ and $s>0$. Given $v=(r,l,s)$ and $v'=(r',l',s')$, the Mukai pairing of $v$ and $v'$ is the product 
\[
<v,v'>=-\int_S v\wedge v'^\vee
\]
where $v'^\vee:=(r',-l',s')$. The lattice $(H^*(S,\ZZ),<,>)$ is known as the Mukai lattice of $S$; it can be equipped with the following pure Hodge structure of weight two:
\begin{align*}
(H^*(S))^{2,0}&:=H^{2,0}(S) \\
(H^*(S))^{1,1}&:=H^0(S,\CC)\oplus H^{1,1}(S)\oplus H^4(S,\CC) \\
(H^*(S))^{0,2}&:=H^{0,2}(S). 
\end{align*}

Given a Mukai vector $v$ with \(v^2\ge2\) and $H$ a $v$-generic polarization on $S$, we denote by  $\M_v(S,H)$ the moduli space of $S$-equivalence classes of $H$-semistable sheaves on $S$ with Mukai vector \(v\). We refer to \cite{PR:OG10} for the definition of $v$-generic polarization; we will work with \(\K3\) surfaces with Picard group of rank one, where any polarization is \(v\)-generic. We denote by $\M_v^s(S,H)$ the moduli space of $H$-stable sheaves over $S$ with Mukai vector $v$, which is an open subset of \(\M_v(S,H)\). 

When the Mukai vector is primitive in the Mukai lattice, the moduli space \(\M_v(S,H)\), if not empty, is known thanks to the work many authors (see \cite{mukai.symplectic,mukai1987moduli,ogrady.weight,yoshioka1998note,yoshioka2001moduli}) to be an IHS manifold of dimension \(2n=v^2+2\), deformation equivalente to the Hilbert scheme \(S^{[n]}\). In this case \(\M^s_v(S,H)=\M_v(S,H)\).

The case we deal with here is when the Mukai vector is non primitive, since it gives rise to the ten dimensional example of IHS manifold we are interested in. The first result in this direction has been found by O'Grady.

\begin{thm}[O'Grady \cite{OG10}] Let $S$ e a projective $\K3$ surface and fix the Mukai vector $v=(2,0,-2)\in H^*(S,\ZZ)$; let $H$ be a $v$-generic polarization on $S$. The moduli space $\M_{10}:=\M_v(S,H)$ admits a symplectic desingularization $\tilde \pi:\widetilde\M_{10}\rightarrow \M_{10}$, which is an IHS manifold of dimension ten and second Betti number at least $24$. 
\end{thm}

By $b_2(\widetilde\M_{10})\ge 24$ one can conclude that $\widetilde\M_{10}$ gives a new deformation class of IHS manifolds: the second Betti number is a deformation invariant, and the Betti numbers of Beauville's examples are 7 and 23, as shown in \cite{beauville.varietes}. In \cite{rapagnettaOG10} Rapagnetta proved that $b_2(\widetilde\M_{10})=24$. 

The manifold $\widetilde\M_{10}$ is actually the only (up to deformation) new example of IHS manifold that can be constructed as a symplectic resolution of a moduli space of semistable sheaves on a projective $\K3$ surface. Also, it arises for several choices of non primitive Mukai vectors.

\begin{thm}\label{thm.desing.Mv} Let $S$ be a projective $\K3$ surface and $v$ a Mukai vector of the form $v=mw$, with $w$ primitive, $m\ge 2$ and $w^2>0$; let $H$ be a $v$-generic polarization on $S$.
\begin{enumerate}
\item If $m=2$ and $w^2=2$, there exists a symplectic desingularization $\tilde\pi_v:\widetilde\M_v(S,H)\rightarrow\M_v(S,H)$ obtained as blow-up of $\M_v(S,H)$ along the singular locus $\Sigma_v:=\M_v\setminus\M_v^s$, taken with reduced structure (Lehn, Sorger \cite{lehnsorger}). 
\item Under the hypotheses of (1), $\widetilde\M_v(S,H)$ is deformation equivalent to $\widetilde\M_{10}$. The pullback morphism 
\[
\tilde\pi_v^*:H^2(\M_v(S,H),\ZZ)\rightarrow H^2(\widetilde\M_v(S,H),\ZZ)
\] 
is injective; as consequence, $H^2(\M_v(S,H),\ZZ)$ carries a pure weight-two Hodge structure and a lattice structure, given by the restriction of the pure weight-two Hodge structure of $H^2(\widetilde\M_v(S,H),\ZZ)$ and of its Beauville-Bogomolov-Fujiki form. Furthermore, there exists a Hodge isometry 
\[
\lambda_v:v^\perp\rightarrow H^2(\M_v(S,H),\ZZ).
\] 
(Perego, Rapagnetta \cite{PR:OG10}).
\item If $m\ge 3$ or $w^2\ge 4$, then the moduli space $\M_v(S,H)$ does not admit any symplectic resolution (Kaledin, Lehn, Sorger \cite{kaledin.lehn.sorger}).
\end{enumerate}
\end{thm}

The following Corollary is \cite[Corollary 2.7]{PR:factoriality}:

\begin{cor}[Perego, Rapagnetta] \label{cor.isom.hodge} Assuming the hypotheses of Theorem \ref{thm.desing.Mv}.(2), the isometry $\lambda_v$ restricts to an isometry 
\[\lambda_v:(v^\perp)^{1,1}:=v^\perp\cap (H^*(S))^{1,1}\xrightarrow{\sim} \Pic(\M_v(S,H)).
\]
\end{cor}

\cite{PR:factoriality} gives a lattice-theoretic description of the second cohomology group of the symplectic resolution \(\widetilde\M_v(S,H)\) of the moduli space \(\M_v(S,H)\), that we recall here. Given a Mukai vector \(v\) as in the hypotheses of Theorem \ref{thm.desing.Mv}.(2), we consider the lattice \((v^\perp\otimes_\ZZ\QQ)\oplus^\perp \QQ\cdot\sigma\), where \(v^\perp\otimes_\ZZ\QQ\) is the lattice endowed with the \(\QQ\)-linear extension of the Mukai pairing, \(\sigma\) is a class with square \(-6\) and the direct sum is defined to be orthogonal. Define \(\Gamma_v\) as the following subgroup (hence lattice) of \((v^\perp\otimes_\ZZ\QQ)\oplus^\perp \QQ\cdot\sigma\):
\begin{equation}\label{eq.Gamma_v}
\Gamma_v:=\Bigl\{\Bigl(\frac{\beta}{2},k\frac{\sigma}{2}\Bigl)\ \Bigl|\ \beta\in (v^\perp)^{1,1},\ k\in\ZZ,\ <\beta,v^\perp>\subseteq 2\ZZ,\ k\in2\ZZ\Leftrightarrow\frac{\beta}{2}\in v^\perp\Bigl\}.
\end{equation}
The lattice \(\Gamma_v\) can be equipped of the following pure Hodge structure of weight two:
\begin{align*}
\Gamma_v^{2,0}&:=(v^\perp)^{2,0} \\
\Gamma_v^{1,1}&:=(v^\perp)^{1,1}\oplus \CC\cdot\sigma \\
\Gamma_v^{0,2}&:=(v^\perp)^{0,2}.
\end{align*}

The following is \cite[Theorem 3.4]{PR:factoriality}.
\begin{thm}[Perego, Rapagnetta]\label{thm.Gamma_v} Assume the hypothesis of Theorem \ref{thm.desing.Mv}. Following the notation introduced above, there exists a Hodge isometry:
\begin{align*}
f_v:\Gamma_v&\xrightarrow{\sim} H^2(\widetilde\M_v(S,H),\ZZ) \\
\Bigl(\frac{\beta}{2},k\frac{\sigma}{2}\Bigl)&\mapsto (\widetilde\pi\circ\lambda_v)\Bigl(\frac{\beta}{2}\Bigl)+\frac{k}{2}\widetilde\Sigma_v.
\end{align*} 
Here, by abuse of notation, we are denoting by \(\widetilde\pi_v\circ\lambda_v\) the \(\QQ\)-linear extension of it.
\end{thm}

Given a Mukai vector $v\in H^*(S,\ZZ)$ that is non primitive, the moduli spaces $\M_v(S,H)$ can happen to be non locally factorial. We recall here the definition of locally and $k$-factorial.

\begin{defn} Let $X$ be a normal projective variety, and let $A^1(X)$ be the group of Weil divisors of $X$, up to linear equivalence; consider the natural inclusion $d:\Pic(X)\rightarrow A^1(X)$ that to any line bundle associates its Weil divisor. $X$ is said to be $k$-factorial if the cokernel of $d$  is $k$-torsion, and it is said to be locally factorial if $d$ is an isomorphism.
\end{defn}

The moduli spaces $\M_v(S,H)$ are either locally factorial or 2-factorial, according to the following criterion, see \cite[Theorem 1.1]{PR:factoriality}:

\begin{thm}[Perego, Rapagnetta] \label{thm.factorial} Assume the notation in Notation \ref{notation.S} and let us consider a Mukai vector of the form $v=2w\in H^*(S,\ZZ)$. Then 
\begin{itemize}
\item $\M_v(S,H)$ is 2-factorial if and only if it exists $\gamma\in (H^*(S))^{1,1}$ such that \\ $<\gamma, w>=1$.
\item $\M_v(S,H)$ is locally factorial if and only if for any $\gamma\in (H^*(S))^{1,1}$ one has $<\gamma, w>\in 2\mathbb Z$.
\end{itemize}
\end{thm}

Notice that, when the moduli space \(\M_v(S,H)\) is locally factorial, one has:
\[
\Gamma_{v_1}\cong v^\perp\oplus^\perp \ZZ\cdot\sigma
\]
as consequence of \cite[Proposition 4.1,(2)]{PR:factoriality}.

\subsection{Lagrangian structure}\label{subsection:Lagrangian}

\begin{notation}\label{notation.S} From now on, $S$ will be a projective $\K3$ surface with $\Pic(S)=\ZZ\cdot H$, where $H$ is ample and $H^2=2$; we call $h:=c_1(H).$
\end{notation}

\begin{rem}\label{rem.conics} Under these hypotheses, the polarization $H$ induces a surjective morphism $f:S\rightarrow|H|^\vee\cong\PP^2$ which has degree 2 and ramifies along a sextic. We will denote by $\iota$ the involution on \(S\) induced by $f$. Furthermore, the pullback $f^*$ induces a bijection among lines of $\PP^2$ and curves in the linear system $|H|$ on $S$, and a bijection among conics in $\PP^2$ and curves in $|2H|$ on $S$.
\end{rem}

We focus here on some particular moduli spaces of sheaves, that are the ones obtained fixing a Mukai vector of the form
\[
v=(0,2h,a).
\] 
This are the models we will use in order to define uniruled divisors on $\OG10$ manifolds. Given this choice of Mukai vector, then there exists a regular morphism: 
\[p_{a}:\M_{(0,2h,a)}(S,H)\rightarrow |2H|
\] 
sending a sheaf in $\M_{(0,2h,a)}(S,H)$ to its Fitting scheme. When $v=(0,2h,a)$ is non primitive (i.e.\ \(a\) is even), we will call \(\widetilde p_a\) the composite morphism:
\[\tilde p_{a}:= p_{a}\circ\tilde\pi:\tilde \M_{(0,2h,a)}(S,H)\rightarrow |2H|
\]
where $\tilde\pi:\widetilde\M_{(0,2h,a)}(S,H)\rightarrow \M_{(0,2h,a)}(S,H)$ is the symplectic desingularization. Thanks to Matsushita's theorem (see \cite{matsushita1999addendum}), $\tilde p_{a}$ is a Lagrangian fibration. 

An easy computation \footnote{This essentially uses the Grothendieck-Riemann-Roch theorem on sheaves in $\M_{(0,2h,a)}(S,H)$, which by definition have first Chern class equal to $2h$.} shows that, given $C\in|2H|$ smooth: \begin{equation}\label{eq.fiber.Pic^d}
p^{-1}_{a}(C)\cong\Pic^{4+a}(C).
\end{equation}
Notice that this implies in particular that any sheaf in $\M_{(0,2h,a)}(S,H)$ on a smooth curve is stable, hence it is a smooth point of the moduli space. We conclude that the moduli spaces $\M_{(0,2h,a)}(S,H)$ and \(\widetilde\M_{(0,2h,a)}(S,H)\) contain as open dense subset the relative Picard scheme $\mathcal J^{4+a}_{|2H|^{sm}}$ of degree \(4+a\), where $|2H|^{sm}\subseteq|2H|$ is the open subset consisting of smooth curves; see Chapter XXI.2 in \cite{ACGH} for a definition of the relative Picard scheme. This property will be crucial in what follows, since it will be used to define divisors in $\M_{(0,2h,a)}(S,H)$.


\section{Uniruled divisors and ampleness}\label{section:uniruled divisors}

Let $S$ be a K3 surface as in Notation \ref{notation.S}. We fix a rational curve $\rho_0\in |H|$ satisfying the following conditions.
\begin{description}
\item[\setword{(a)}{word:condition.a}] Consider the surjective morphism $f:S\rightarrow \PP^2$ discussed in Remark \ref{rem.conics}, and call $\tau'$ the pencil of conics in $\PP^2$ image of the pencil $\tau\subset |2H|$. The line $f(\rho_0)\subseteq \PP^2$ is not tangent to any of the three singular conics in the pencil $\tau'$. 
\item[\setword{(b)}{word:condition.b}] Let \(\delta\subseteq\PP^2\) be the branch locus of \(f:S\to\PP^2\), which is a sextic. The line $f(\rho_0)\subseteq \PP^2$ is not tangent to the conics of $\tau'$ which are tangent to \(\delta\).
\end{description}
Notice that the conditions above are open conditions. We call $\rho\xrightarrow{\nu} \rho_0$ the normalization of $\rho_0$.
\\

We want to define uniruled divisors on IHS manifolds of \(\OG10\) type. To start, we fix two models of \(\OG10\) manifolds given by Theorem \ref{thm.desing.Mv}.(2). We will call:
\begin{equation}\label{eq.def.v1.v2}
v_1:=(0,2h,4),\ v_2:=(0,2h,2)\in H^*(S,\ZZ).
\end{equation}
The vectors $v_1$ and $v_2$ are Mukai vectors, according to the definition given in Section \ref{section:moduli}. We will denote by $\M_{v_1}$ and $\M_{v_2}$ the moduli spaces $\M_{v_1}(S,H)$ and $\M_{v_2}(S,H)$ respectively. The Mukai vectors ${v_1}$ and $v_2$ are non primitive in the Mukai lattice, and they satisfy the hypotheses of Theorem \ref{thm.desing.Mv}.(2). It follows that there exist symplectic desingularizations: 
\[
\tilde\pi_1:\widetilde\M_{v_1}\rightarrow\M_{v_1},\ \tilde\pi_2:\widetilde\M_{v_2}\rightarrow\M_{v_2}
\]  
such that $\widetilde\M_{v_1}$ and $\widetilde\M_{v_2}$ are IHS manifolds of $\OG10$ type. Notice that these moduli spaces contain the following relative Jacobian manifolds on $|2H|^{sm}\subseteq |2H|$ (cf. the discussion in Subsection \ref{subsection:Lagrangian}): 
\[
\mathcal J^8_{|2H|^{sm}}\subseteq \M_{v_1}
\]
\[
\cJ^6_{|2H|^{sm}}\subseteq \M_{v_2}.
\]
This geometrical property will be used to define the uniruled divisors.

\subsection{Two examples of uniruled divisors on \(\widetilde\M_{v_1}\)}\label{subsection:divisor1}

We start with the moduli space $\M_{v_1}$.

\begin{defn}\label{def.D} Let $D_1\subset\M_{v_1}$ be the divisor defined as the closure in $\M_{v_1}$ of the locus in $\mathcal J^8_{|2H|^{sm}}$ consisting of sheaves of the form $i_*\mathcal O_C(4r+p_1+p_2+p_3+p_4)$, where $C\in|2H|^{sm}$, $i:C\hookrightarrow S$ is the inclusion, $p_1,..,p_4\in C$ and $r\in C\cap\rho_0$.
\end{defn}

The multiplicities of the points in the definition of $D_1$ are chosen so that on each smooth curve \(C\subseteq S\) one has \(\cO_C(4r+p_1+p_2+p_3+p_4)\in J^8(C)\), and the line bundle varies in a locus of codimension one in \(J^8(C)\), since \(r\in C\cap \rho_0\) which consists of at most four distinct points. 

\begin{rem}\label{rem.D.schematic} 
We define here a schematic structure on $D_1$, that we will use in the computations that follow. Let $\cF$ be a quasi-universal family for the smooth locus $\M^s_{v_1}$ of the moduli space, i.e.\ $\cF\in \Coh(S\times\M^s_{v_1})$ such that, for any $[F]\in \M^s_{v_1}$, one has $\cF|_{S\times [F]}\cong F$ as a sheaf on $S$. We will use $\cF$ to endow $D_1$ with a schematic structure: the basic idea is to define the divisor $D_1$ as the closure in $\M_{v_1}$ of the support of the first derived functor of the pushforward via $p_{\M^s_{v_1}}:S\times\M^s_{v_1}\rightarrow \M^s_{v_1}$ of the quasi-universal family twisted by a suitable line bundle, that will encode the point with multiplicity 4 on the rational curve $\rho_0$ appearing in the definition of $D_1$.

Let $\rho_0\in |H|$ be the rational curve defining $D_1$, and $\rho\xrightarrow{\nu}\rho_0$ its normalization. Consider the following commutative diagram:
\[
\begin{tikzcd}
\rho\times S\times \M^s_{v_1}\arrow{r}{p_{S\times\M^s_{v_1}}} \arrow{d}{p_{\rho\times\M^s_{v_1}}} & S\times \M^s_{v_1} \arrow{d}{p_{\M^s_{v_1}}}\\
\rho\times\M^s_{v_1} \arrow{r}{p_{\M^s_{v_1}}} & \M^s_{v_1}
\end{tikzcd}
\]
where all the maps are projections to the space written in the subscript of the map. 
Note that the pullback sheaf $\hat{\cF}:=p_{S\times\M^s_{v_1}}^*\cF\in\Coh(\rho\times S\times\M^s_{v_1})$ is supported on the incidence variety: 
\[
I_{\rho\times S\times \M^s_{v_1}}=\{(r,x,[F])\in \rho\times S\times\M^s_{v_1}\ |\ \nu(r),x\in \Supp(F)\}\subseteq \rho\times S\times\M^s_{v_1}.
\]
Now on, with $\hat\cF$ we will always mean the restriction of $\hat\cF$ to its support. Consider also the incidence variety:
\[
I_{\rho\times \M^s_{v_1}}=\{(r,[F])\in\rho\times\M^s_{v_1}\ |\ \nu(r)\in \Supp(F)\}\subseteq \rho\times\M^s_{v_1}.
\]
Finally, consider the following section of $p_{\rho\times\M^s_{v_1}}:\rho\times S\times \M_{v_1}^s\to\rho\times\M^s_{v_1}$:
\begin{align*}
\sigma: \rho\times\M^s_{v_1}&\rightarrow \rho\times S\times\M^s_{v_1} \\
(r,[F])&\mapsto (r,\nu(r),[F])
\end{align*}
and call 
\[
R:=\sigma(I_{\rho\times \M^s_{v_1}})\subseteq I_{\rho\times S\times \M^s_{v_1}}\subseteq \rho\times S\times \M^s_{v_1}.
\]
Since $R$ is a Weil divisor on $I_{\rho\times S\times \M^s_{v_1}}$, we can consider the following sheaf:
\[
\hat\cF\otimes\cI_{R}^{\otimes 4}\in\Coh(I_{\rho\times S\times \M^s_{v_1}}).\]
We define: 
\[
D''_1:=\Supp\Bigl(R^1p_{\rho\times\M^s_{v_1},*}\bigl(\hat\cF\otimes\cI_{R}^{\otimes 4}\bigl)\Bigl)
\]
which is a divisor on $I_{\rho\times\M_v^s}$; here the support is the schematic support of the sheaf. Finally, we consider the closure of its image via the projection \(p_{\M^s_{v_1}}\):
\[
D'_1:=\overline{p_{\M_{v_1}}(D''_1)}^{\M_{v_1}}\subseteq \M_{v_1}.
\]
From $R^2p_{\rho\times\M^s_{v_1},*}\hat\cF\otimes\cI_{R}^{\otimes 4}=0$ it follows that, for every $(r,[F])\in I_{\rho\times\M^s_{v_1}}$:
\begin{align*}
\bigl(R^1p_{\rho\times\M^s_{v_1},*}\hat\cF\otimes\cI_{R}^{\otimes 4}\bigl)_{(r,[F])}&\cong H^1((r,\Supp(F),[F]),\hat\cF\otimes\cI_{R}^{\otimes 4}|_{(r,\Supp(F),[F])}) \\
&\cong H^1(\Supp(F),F(-4\nu(r)))
\end{align*}
hence $(r,[F])\in I_{\rho\times\M^s_{v_1}}$ is in $D_1''$ exactly when $h^1(F(-4\nu(r))> 0$. When $\Supp(F)=C\in |2H|^{sm}$, thanks to the Riemann-Roch theorem on a smooth curve one has: 
\[
\chi(C,F(-4\nu(r)))=h^1(C,F(-4\nu(r))-h^0(C,F(-4\nu(r))=0
\] 
where we used that the curves in \(|2H|\) have genus 5 and that \(F(-4\nu(r))\) is a line bundle of degree 4 on the curve. Notice that this implies the same equality for any $\Supp(F)\in|2H|$, since the Euler characteristic is constant on flat families.  

In particular we get that, from a set-theoretic point of view:
\[
D'_1|_{\cJ^8_{|2H|^{sm}}}=\{[F]\in \M_{v_1}\ |\ \exists r\in\rho_0:\ h^0(F(-4r))>0\}=D_1|_{\cJ^8_{|2H|^{sm}}}.
\]
Since the divisors $D_1$ and $D_1'$ coincide on $\cJ^8_{|2H|^{sm}}$, we can give to $D_1|_{\cJ^8_{|2H|^{sm}}}$ the schematic structure of $D'_1|_{\cJ^8_{|2H|^{sm}}}$, and then to $D_1$ the schematic structure given by the closure of $D_1|_{\cJ^8_{|2H|^{sm}}}$ in $\M_{v_1}$. 

Observe that $p_{\M_{v_1}}$ is generically $4:1$ on $I_{\rho\times \M_{v_1}}$, that is the reason why we are defining $D'_1$ as the image of $D''_1$ via $p_{\M_{v_1}}$, and not as its pushforward: we want $D_1$ to have a reduced schematic structure. 
\end{rem}

The geometrical intuition behind the definition of the divisor $D_1$ is given by the following result.

\begin{lem}\label{lem.D.uniruled} The divisor $D_1$ of Definition \ref{def.D} is uniruled. 
\end{lem}
\begin{proof}
There exists a rational dominant map 
\[
\phi:\rho_0\times \Sym^4(S)\dashrightarrow D_1
\] 
defined as follows: for a generic $\xi=\bigl(r,[p_1,p_2,p_3,p_4] \bigl)\in\rho_0\times \Sym^4(S)$, define:
\[
\phi(\xi):=i_*\mathcal O_{C_\xi}(4r+p_1+p_2+p_3+p_4)\in J^8(C_\xi)
\] 
where $i:C_\xi\hookrightarrow S$ is the unique curve of $|2H|$ passing through the five points in $S$ given by $\xi$. Notice that such a curve is actually unique and smooth for a generic choice of $\xi\in \rho_0\times \Sym^4(S)$: it is the pullback via $f:S\rightarrow \PP^2$ of the unique smooth conic passing through the five (general) points $f(r),f(p_1),f(p_2),f(p_3),f(p_4)\in \PP^2$, see Remark \ref{rem.conics}.
\end{proof}

\begin{rem}\label{rem.D1eD1'} We spend here a few words about how the divisors $D_1$ and $D_1'$ introduced in Remark \ref{rem.D.schematic} are related. This will be used only partially in what follows, but we believe it is useful to have a clearer picture in mind. Firstly, we observe that by definition $D_1\subseteq D_1'$. 

In Remark \ref{rem.D.schematic}, we have already noticed that $[F]\in D_1\cap \cJ^8_{|2H|^{sm}}$ if and only if there exists $r\in \Supp(F)\cap \rho_0$ such that $h^0(F(-4r))>0$, and similarly for \([F]\in D_1'\cap\cJ^8_{|2H|^{sm}}\). This is indeed true for any $[F]\in D'_1\cap \M^s_{v_1}$, since as noticed in Remark \ref{rem.D.schematic} the divisor $D_1''\subseteq I_{\rho\times\M_{v_1}^s}$ is the locus of the pairs $(r,[F])\in I_{\rho\times\M_{v_1}^s}$ such that $h^0(F(-4\nu(r)))>0$. The function $h^0$ is upper semicontinuous on flat families (cf. \cite{huy.lehn});  we conclude that any sheaf \([F]\) in $D'_1$ and $D_1$  is such that \(h^0(F(-4r))>0\) for some \(r\in \Supp(F)\cap\rho_0\). Observe that this is only a necessary condition for a sheaf \([F]\) to be in \(D_1\) and in \(D_1'\).

Let $\Sym^2|H|\subseteq |2H|$ be the locus of reducible curves, which is a divisor in $|2H|$ contained in \(|2H|\smallsetminus|2H|^{sm}\). Given the projection $p:\M_{v_1}\rightarrow |2H|$ introduced in Subsection \ref{subsection:Lagrangian}, consider the divisor $p^*\Sym^2|H|\subseteq\M_{v_1}$. Take $[F]\in p^*\Sym^2|H|$ general; this is a stable sheaf, hence it is in $\M^s_{v_1}$; we want to show that $[F]\in D_1'\cap \M^s_{v_1}$, thanks to the characterization of the points in this intersection given above. The sheaf $[F]$ is  supported on a reducible curve $C=C_1\cup C_2$, with $C_i\in|H|$ smooth and $C_1\cap C_2=\{n_1,n_2\}$, since $[F]\in p^*\Sym^2|H|$ is general; also, since $[F]$ is a stable sheaf, we have $\deg(F|_{C_1})=\deg(F|_{C_2})=4$. Now, take $r\in C_1\cap\rho_0$; a non-zero section in $H^0(S,F(-4r))$ is constructed as follows: we glue the zero section on $C_1$ and a non-zero section in $H^0(C_2,F|_{C_2})$ vanishing on $n_1,n_2$, which exists because $\deg(F|_{C_2}(-n_1-n_2))=\deg(F|_{C_2})-2=2$ and $g(C_2)=2$. We conclude that $[F]\in D_1'\cap \M^s_{v_1}$, hence the divisor $D_1'$ contains 
$p^*\Sym^2|H|$.

The divisor $D_1$ is an irreducible divisor, because it is the closure in \(\M_{v_1}\) of an irreducible divisor on $\cJ^8_{|2H|^{sm}}$. We conclude that $D_1\subsetneq D_1'$, and moreover:
\[
D_1'\supseteq D_1+kp^*\Sym^2|H|.
\]
with $k\ge1$.
\end{rem}

By definition, \(D_1\) is an effective Weil divisor on \(\M_{v_1}\). Given the symplectic desingularization \(\widetilde\pi_1:\widetilde\M_{v_1}\to\M_{v_1}\), we consider the strict transform of the divisor:
\[
\widetilde D_1\in \Pic(\widetilde \M_{v_1})
\]
which is an effective uniruled divisor on \(\widetilde \M_{v_1}\), ruled by irreducible and reduced curves. 

\begin{rem}\label{rem.pullback.strict.trans} The pullback divisor $D_1^*:=\tilde\pi_1^*D_1\subseteq \widetilde \M_{v_1}$ is also a uniruled divisor, since it is a combination of $\widetilde D_1$ and the exceptional divisor \(\widetilde \Sigma_1\) of the desingularization, which is uniruled. Anyway, when the singular locus \(\Sigma_1\) of \(\M_{v_1}\) is contained in \(D_1\) with multiplicity \(m\ge 2\), then $D^*_1=\widetilde D_1+m\widetilde\Sigma_1$, where \(\widetilde D_1\) is the strict transform of \(D_1\) via \(\widetilde\pi_1\) and \(m\ge 2\). We get that $D_1^*$ is ruled by non reduced curves, whose irreducible components are the reduced curves ruling \(\widetilde D_1\) and the non reduced curves ruling \(m\widetilde\Sigma_1\), with \(m\ge 2\); in this case we wouldn't be able to apply Corollary \ref{cor.defo.curves} to the curves ruling \(D_1^*\), that is our main tool to get the existence of ample uniruled divisors in some connected components of $\mathfrak M_{\OG10}^{pol}$. We will come back to this point later.
\end{rem}

We call:
\[
B_1:=\widetilde D_1+\widetilde\Sigma_1\in \Pic(\widetilde\M_{v_1})
\]
which is an effective uniruled divisor on \(\widetilde \M_{v_1}\); it is ruled by reducible but reduced curves, whose irreducible components are the reduced curves ruling \(\widetilde D_1\) and the reduced curves ruling \(\widetilde \Sigma_1\).

\subsection{Two examples of uniruled divisors on \(\widetilde\M_{v_2}\)}\label{subsection:divisor2} We pass to the moduli space $\M_{v_2}$. We start defining a divisor \(D_2\subseteq\M_{v_2}\) adapting the definition of $D_1$ to the Mukai vector $v_2$.

\begin{defn}\label{def.D2} Let $D_2\subset \M_{v_2}$ be the divisor defined as the closure in $\M_{v_2}$ of the locus in $\mathcal J^6_{|2H|^{sm}}$ consisting of sheaves of the form $i_*\cO_C(2r+p_1+...+p_4)$, with $C\in|2H|^{sm}$, $i:C\hookrightarrow S$, $p_1,...,p_4\in C$ and $r\in C\cap\rho$.
\end{defn}

\begin{lem} The divisor $D_2$ of Definition \ref{def.D2} is uniruled. 
\end{lem}
\begin{proof}
The proof goes exactly as the proof of Lemma \ref{lem.D.uniruled}, only slightly changing the coefficients in the definition of the map $\phi$ in the proof. 
\end{proof}

\begin{rem}\label{rem.D2.schematic} The schematic structure of $D_2$ is given as that of $D_1$ in Remark \ref{rem.D.schematic}: following the same notation, one only needs to consider a quasi-universal family $\cF$ of the moduli space $\M^s_{v_2}$, and then the sheaf $\hat\cF\otimes\cI_R^{\otimes 2}\in \Coh(I_{\rho\times S\times\M^s_{v_2}})$. After these modifications, Remark \ref{rem.D.schematic} applies to the divisors $D''_2$, $D_2'$ and then $D_2$ without any further change.
\end{rem} 

The divisor $D_2$ is by definition an effective Weil divisor. Given the symplectic desingularization
$\tilde\pi_2:\widetilde\M_{v_2}\rightarrow\M_{v_2}$, we consider the strict transforms of the divisor: 
\[
\widetilde D_2\in \Pic(\widetilde \M_{v_2}).
\]
Furthermore, we call:
\[
B_2:=\widetilde D_2+\widetilde \Sigma_2\in  \Pic(\widetilde \M_{v_2}).
\]
Obviously, the divisors $\widetilde D_2$ and \(B_2\) are effective uniruled divisors; as for the divisors \(\widetilde D_1\) and \(B_1\) defined in the previous section, the divisors \(\widetilde D_2\) and \(B_2\) are ruled by reduced curves.

\subsection{Existence of ample uniruled divisors}\label{subsection:ampleness}

After we have defined some uniruled divisors on IHS manifolds, we want to check that these divisors are ample. In this section we will develop a method to check the positivity of a divisor, that we will see in a moment (see Remark \ref{rem.need.positivity}) to be strictly related to our problem. The method we are going to present will be applied in the following sections to the divisors we have defined in Subsection \ref{subsection:divisor1} and Subsection \ref{subsection:divisor2}. 

We start recalling the following definition.

\begin{defn}
Let \(X\), \(X'\) be compact complex manifolds and  \(L\in\Pic(X)\), \(L'\in\Pic(X')\) line bundles on them. We say that the pair \((X,c_1(L))\) is deformation equivalent to the pair \((X',c_1(L'))\) if there exist a smooth proper morphism \(\pi:\cX\rightarrow S\) with \(S\) connected complex variety, a line bundle \(\cL\in\Pic(\cX)\) and two points \(s,s'\in S\) such that \((\cX_s,c_1(\cL_s))\cong (X,c_1(L))\) and \((\cX_{s'},c_1(\cL_{s'}))\cong (X',c_1(L'))\).
\end{defn}

A crucial point of the strategy we are going to present is the following easy remark. 

\begin{rem}\label{rem.need.positivity} Let \(X\) be an IHS manifold and assume we have $D\in\Pic(X)$ effective and with $q_X(D)>0$. Then there exists a deformation of $(X,D)$ where the divisor is ample. Indeed, we can consider a small deformation $(X',D')$ of $(X,D)$ such that $\Pic(X')\cong \ZZ$ and $D'$ is again effective. Since $q_{X'}(D')=q_X(D)>0$ we deduce that $X'$ is projective (see \cite[Theorem 3.11]{huybrechts}), and by $\Pic(X')\cong \ZZ$ we conclude that $D'$ is ample. 

Now assume that $D$ is uniruled, ruled by reduced curves. We will see with Corollary \ref{cor.defo.curves} that on the deformation $X'$ of \(X\) there exists a uniruled divisor $D''$, whose cohomology class is a multiple of the one of $D'$. In conclusion we obtain an ample uniruled divisor on $X'$, proportional to the divisor $D'$ obtained as deformation of $D$ on $X$. 
\end{rem} 

As consequence of Remark \ref{rem.need.positivity}, we are interested to answer the following question.
\begin{question}\label{quest.q(D)} Given a (effective and uniruled) divisor on $X$, with $X$ IHS manifold of $\OG10$ type, how to compute the square of it with respect to the Beauville-Bogomolov-Fujiki form $q_X$ on \(X\)?
\end{question}

We present here an answer to this question in the case of \(\OG10\) manifolds that are desingularized moduli spaces of semistable sheaves $\widetilde \M_v(S,H)$ with $v=2w$, $w$ primitive and $w^2=2$, and when the divisor on \(\widetilde \M_v(S,H)\) is the pull-back a divisor \(D\) on \(\M_v(S,H)\) via the symplectic desingularization \(\widetilde\pi_v:\widetilde\M_{v}(S,H)\to\M_v(S,H)\). By Theorem \ref{thm.desing.Mv}.(2), given a divisor $D\in\Pic( \M_{v}(S,H))$ one has:
\[
q_{10}(\widetilde\pi^*_v D)=q_{10}(D)=(\lambda_v^{-1}D)^2
\] 
where the square on the right is with respect to the Mukai pairing introduced in Section \ref{section:moduli}. As consequence, our goal is to understand the class $\lambda_v^{-1}D$ inside the Mukai lattice; once this is done, the Mukai square $(\lambda_v^{-1}D)^2$ will follow from an easy computation.

Assume the setting in Notation \ref{notation.S}, and that $v$ is as in Theorem \ref{thm.desing.Mv}.(2). Since $\Pic(S)\cong\ZZ$, we have $\rk (v^\perp)^{1,1}=\rk\Pic(\M_v(S,H))=2$; in what follows, we will always denote by $\{e,f\}$ a $\ZZ$-basis of $(v^\perp)^{1,1}$. In order to answer Question \ref{quest.q(D)}, we will follow the following strategy.

\begin{strategy}\label{strategy} Let us assume the hypotheses of Theorem \ref{thm.desing.Mv}.(2), and take $D\in\Pic(\M_v(S,H))$. In order to compute $q_{10}(\tilde\pi^*D)$ we will:
\begin{enumerate}
\item Define two numerically independent curves in $\M_v(S,H)$.
\item Compute the geometric intersection in $\M_v(S,H)$ of the divisors $\lambda_v(e)$, $\lambda_v(f)$ with the curves defined in point (1). 
\item Compute the geometric intersection inside $\M_v$ of $D$ with the curves defined in point (1).
\item Compare what computed in points $(2)$ and $(3)$ to write $\lambda_v^*D$ in terms of the basis $\{e,f\}$ of $(v^\perp)^{1,1}$, and compute $(\lambda_v^*D)^2=q_{10}(\tilde\pi^*D)$.
\end{enumerate}
\end{strategy}

\begin{rem} By the words "numerically independent" in Strategy \ref{strategy}.(1) we mean two curves such that the system of two equations and two variables that we get comparing the intersections in (2) and (3) has a unique solution.
\end{rem}

We will devote Section  \ref{section:generators Mukai and intersection} and Section \ref{section:class of the divisors} to the application of Strategy \ref{strategy} to the divisors  defined in Subsection \ref{subsection:divisor1} and Subsection \ref{subsection:divisor2} on the IHS manifolds $\widetilde\M_{v_1}$ and $\widetilde\M_{v_2}$. Observe that these divisors are not defined as pull-back of divisors on the singular moduli spaces, but they are obtained by pull-back divisors \(\widetilde\pi_1^* D_1\) and \(\widetilde\pi^*_2 D_2\) by adding or removing the exceptional divisors of \(\widetilde\M_{v_1}\) and \(\widetilde\M_{v_2}\) with suitable multiplicity; for this reason, the computation of the class (and the square) of \(\widetilde \pi^*_1 D_1\) and \(\widetilde\pi^*_2 D_2\) is a necessary step. In  Section \ref{section:generators Mukai and intersection} we will develop the computations necessary for part (2) of the Strategy, and in  Section \ref{section:class of the divisors} the computations for part (3) and part (4). We conclude this section with the definition of the curves as in part (1) of the Strategy.

\subsection{Definition of the curves}\label{subsection:def curves} 

Consider:
\begin{itemize}
\item a smooth curve $\gamma\in |2H|$, a fixed point $p_0\in\gamma$, the inclusion $i:\gamma\times\gamma\hookrightarrow S\times\gamma$ and the diagonal $\Delta\subset \gamma\times\gamma$. 
\item four general points $q_1,...,q_4\in S$, their conjugate points $q_{i+4}:=\iota (q_i)$ under the natural involution $\iota$ on $S$, for $i=1,..,4$  (cfr. Remark \ref{rem.conics}), the pencil $\tau\subset |2H|$  defined by taking $\{q_1,...,q_8\}$ as base points, the inclusion  $j:\mathscr C\hookrightarrow S\times\tau$ of the  universal curve of $\tau$. 
\end{itemize}

We start working with the moduli space $\M_{v_1}$.

\begin{defn}\label{def.curves1} Using the notation of above, we define:
\begin{align*}
\mathcal E_{\Gamma_1}&:=i_*\mathcal O_{\gamma\times\gamma}(7p_0\times\gamma+\Delta)\ \in\Coh(S\times\gamma) \\
\mathcal E_{\mathcal T_1}&:=j_* \mathcal O_{\mathscr C}\bigl(2(q_1\times\tau+q_2\times\tau+q_3\times\tau+q_4\times\tau)\bigl)\ \in\Coh(S\times\tau).
\end{align*}

We will call $\Gamma_1$ and $\cT_1$ the curves in $\M_{v_1}$ defined by $\cE_{\Gamma_1}$ and $\cE_{\cT_1}$ respectively.
\end{defn}

The sheaves \(\cE_{\Gamma_1}\) and \(\cE_{\cT_1}\) are 1-dimensional flat families of \(S\)-semistable sheaves on \(S\) with Mukai vector \(v_1\), hence they define curves in \(\M_{v_1}\). The families $\cE_{\Gamma_1}$ and $\cE_{\cT_1}$ have been chosen in order to parametrize a curve \(\Gamma_1\) that is "vertical" with respect to the fibration $p:\M_{v_1}\rightarrow |2H|$ introduced in Subsection \ref{subsection:Lagrangian}, i.e. a curve that is contained in the fiber $J^8(\gamma)\subseteq \M_{v_1}$ (it is a copy of $\gamma$ inside its Jacobian of degree 8), and a curve \(\cT_1\) that is "horizontal" in $\M_{v_1}$, i.e. a section of the pencil $\tau\subset |2H|$. This will give the numerical independence mentioned in point (1) of Strategy \ref{strategy}.

We conclude with the definition of two curves in the moduli space $\M_{v_2}$. 

\begin{defn}\label{def.curves2} Using the notation of above, we define:
\begin{align*}
\mathcal E_{\Gamma_2}&:=i_*\mathcal O_{\gamma\times\gamma}(5p_0\times\gamma+\Delta)\ \in\Coh(S\times\gamma). \\
\mathcal E_{\mathcal T_2}&:=j_* \mathcal O_{\mathscr C}(3q_1\times\tau+q_2\times\tau+q_3\times\tau+q_4\times\tau)\ \in\Coh(S\times\tau).
\end{align*}
We will call $\Gamma_2$ and $\mathcal T_2$ the curves in $\M_{v_2}$ defined by $\mathcal E_{\Gamma_2}$ and $\mathcal E_{\mathcal T_2}$ respectively.
\end{defn}

\section{Generators of the Mukai lattice}\label{section:generators Mukai and intersection}

In this section we will deal with point (2) of Strategy \ref{strategy}, applied to the moduli spaces $\M_{v_1}$ and $\M_{v_2}$. We will treat separately the case of Mukai vector $v_1$ and the case of Mukai vector $v_2$. 

\subsection{The $v_1$-case}\label{subsection:v1 intersection curves}
We start choosing a $\ZZ$-basis for $(v_1^\perp)^{1,1}$. Observe that:
\[
(v_1^\perp)^{1,1}=\{(a,ah,c),\ a,c\in\mathbb Z\}.
\]
We choose the following generators:
\begin{equation}\label{eq.gen.vperp1}
e:=(1,h,0), \ \ f:=(0,0,1).
\end{equation}

Let $\Gamma_1,\cT_1\subset\M_{v_1}$ be the curves introduced in Definition \ref{def.curves1}. The following result answers to point $(2)$ of Strategy \ref{strategy}.

\begin{prop}\label{prop.chern} Let $e,f\in (v_1^\perp)^{1,1}$ be as in \eqref{eq.gen.vperp1}, and consider the isometry $\lambda_{v_1}:(v_1^\perp)^{1,1}\tilde\rightarrow \Pic(\M_{v_1})$ introduced in Corollary \ref{cor.isom.hodge}. One has the following intersections:
\begin{enumerate}
\item $\lambda_{v_1}(e)\cdot \Gamma_1=-5$,\ \ $\lambda_{v_1}(f)\cdot \Gamma_1=0$.
\item $\lambda_{v_1}(e)\cdot \mathcal T_1=-5$,\ \ $\lambda_{v_1}(f)\cdot \mathcal T_1=1$.
\end{enumerate}
\end{prop}

A proof of Proposition \ref{prop.chern} is presented in Subsection \ref{subsection:proof.chern}. The fundamental tool that we will use is Theorem \ref{thm.huylehn} below, which is Theorem 8.1.5 in \cite{huy.lehn}. Before stating the theorem, we need to translate some results we have on $\M_{v_1}$ in terms of classes in the Grothendieck group $K(S)$ of $S$.
\\

The Mukai vector of a coherent sheaf defines an homomorphism $\mathfrak v:K(S)\rightarrow H^*(S,\ZZ)$, sending the class of a sheaf to its Mukai vector. As consequence, one can define the usual moduli spaces starting from a class $c\in K(S)$: we will call $\M_c:=\M_{\mathfrak v(c)}$. Since $\chi(E,F)=-<\mathfrak v(E),\mathfrak v(F)>$, given a class $c\in K(S)$ the Mukai vector homomorphism restricts to the orthogonal of $c$ with respect to the Euler pairing: we get $c^\perp\rightarrow \mathfrak v(c)^\perp\subseteq H^*(S,\ZZ)$. If $\mathfrak v^\vee:K(S)\rightarrow H^*(S,\ZZ)$ is the dual Mukai vector homomorphism $c\mapsto \mathfrak v(c)^\vee$, we will call 
\[
\lambda_c:=\lambda_{\mathfrak v(c)}\circ\mathfrak v^\vee|_{c^\perp}:c^\perp\rightarrow H^2(\M_v,\ZZ)
\] 
where $\lambda_{\mathfrak v(c)}$ is the isometry of Theorem \ref{thm.desing.Mv}.(2). The following Lemma is to translate $(v^\perp)^{1,1}$ in $c^\perp$. 

\begin{lem}\label{lemma.perpperp} Following the notation above, let $\alpha\in K(S)$. Then $\alpha\in \{1,h,h^2\}^{\perp\perp}$ if and only if $c_1(\alpha)\in \NS(S)$.
\end{lem}
\begin{proof}
Let $\beta\in K(S)$; then $\beta\in\{1,h,h^2\}^\perp$ if and only if:
\begin{itemize}
\item $\chi(\beta,1)=\int \ch(\beta)\ch(\mathcal O_S)\td(S)=2ch_0(\beta)+ch_2(\beta)=0$, i.e.\ $ch_2(\beta)=-2ch_0(\beta)$;
\item $\chi(\beta,h)=\int \ch(\beta)\ch(\mathcal O_S(H))\td(S)=ch_0(\beta)+ch_1(\beta)\cdot h=0$, i.e.\ $ch_1(\beta)\cdot h=-ch_0(\beta)$;
\item $\chi(\beta,h^2)=\int \ch(\beta)\ch(\mathcal O_S(2H))\td(S)=4ch_0(\beta)-2ch_1(\beta)\cdot h=0$, i.e.\ $ch_1(\beta)\cdot h=-2ch_0(\beta)$.
\end{itemize}
Combining the last two points, we get $\ch(\beta)=(0,c_1(\beta),0)$ with $c_1(\beta)\cdot h=0$, i.e.\ $c_1(\beta)\in H^2(S,\mathbb Z)\cap(H^{2,0}(S)\oplus H^{0,2}(S))$. It follows that $\alpha\in\{1,h,h^2\}^{\perp\perp}$ if for any $b\in H^2(S,\mathbb Z)\cap(H^{2,0}(S)\oplus H^{0,2}(S))$ one has 
\[
\int_S (0,b,0)\ch(\alpha)\td(S)=c_1(\alpha)\cdot b=0,
\] 
meaning $c_1(\alpha)\in \NS(S)$.
\end{proof}

\begin{rem} By definition $(v^\perp)^{1,1}=v^\perp\cap (H^0(S,\CC)\oplus \NS(S)_\CC\oplus H^2(S,\CC))$, i.e.\ $a=(a_0,a_1,a_2)\in v^\perp$ belongs to $(v^\perp)^{1,1}$ if and only if $a_1$ is the first Chern class of a line bundle. In other words, Lemma \ref{lemma.perpperp} says that the Mukai vector homomorphism restricts to an homomorphism 
\[
\lambda_c:c^\perp_h:=c^\perp\cap \{1,h,h^2\}^{\perp\perp}\rightarrow (v^\perp)^{1,1}.
\]
\end{rem}

\begin{defn} Let $C$ be a curve, $\mathcal E_C\in \Coh(S\times C)$ and $\pi_s:S\times C\rightarrow S$, $\pi_C:S\times C\rightarrow C$ the projections. We define the homomorphism $\lambda_{\mathcal E_C}:K(S)\rightarrow \Pic(C)$ to be the composition of the following homomorphisms:
\[
\begin{tikzcd}
K(S) \arrow{r}{\pi_S^*} & K(S\times C) \arrow{r}{[\mathcal E_C]\cdot} & K(S\times C) \arrow{r}{\pi_{C,!}} & K(C) \arrow{r}{det} & \Pic(C).
\end{tikzcd}
\]
\end{defn}

\begin{thm}[Theorem 8.1.5 in \cite{huy.lehn}]\label{thm.huylehn}
Following the notation above, let $\mathcal E_C$ be a flat family of semistable sheaves of class $c$ parametrized by a curve $C$, with classifying morphism $\phi_{\mathcal E_C}:C\rightarrow \M_c$. Then the following diagram commutes:
\[
\begin{tikzcd}
c^{\perp}_h \arrow[hook]{r}{\lambda_c} \arrow[hook]{d} & \Pic(\M_c) \arrow{d}{\phi_{\mathcal E_C}^*} \\
K(S) \arrow{r}{\lambda_{\mathcal E_C}} & \Pic(C).
\end{tikzcd}
\]
\end{thm}

In other words, given $\alpha\in c^\perp_h$ and a curve $C\subset \M_c$ with universal family $\mathcal E_C$, from Theorem \ref{thm.huylehn} it follows: 
\[
\lambda_c(\alpha)\cdot C=\deg \bigl(\lambda_{\mathcal E_C}(\alpha)\bigl).
\]

\begin{rem}\label{rem.E.F} 
We want to apply Theorem \ref{thm.huylehn} to compute the intersection of $\lambda_{v_1}(e)$ and $\lambda_{v_1}(f)$ with the curves $\Gamma$ and $\cT$. Since the morphism $\lambda_c$ is the composition of $\mathfrak v^\vee$ and $\lambda_{\mathfrak v(c)}$, we need to choose generators $[E],\ [F]\in c^\perp_h$ such that $\mathfrak v(E)=e^\vee=(1,-h,1)$ and $\mathfrak v(F)=f^\vee=(0,0,1)$.
\end{rem}

\subsection{Proof of Proposition \ref{prop.chern}}\label{subsection:proof.chern}

We are going to prove Proposition \ref{prop.chern}, applying Theorem \ref{thm.huylehn} to $[E],[F]\in K(S)$ as in Remark \ref{rem.E.F} and $\mathcal E_{\Gamma_1}$, $\mathcal E_{\mathcal T_1}$ as in Definition \ref{def.curves1}.

\subsubsection{Intersections $1$ of Proposition \ref{prop.chern}} 

Following the notation above, we want to compute here the intersections $\lambda_c(e)\cdot \Gamma_1$ and $\lambda_c(f)\cdot \Gamma_1$. Let $\pi_{S}:S\times \gamma\rightarrow S$ and $\pi_{\gamma}:S\times \gamma\rightarrow \gamma$ be the two projections. Thanks to Theorem \ref{thm.huylehn} we need to compute
\[
c_1\bigl(\pi_{\gamma,!}(\mathcal E_{\Gamma_1}\otimes\pi_S^*E)\bigl)=\ch_1\bigl(\pi_{\gamma,!}(\mathcal E_{\Gamma_1}\otimes\pi_S^*E)\bigl)
\]
and
\[
c_1\bigl(\pi_{\gamma,!}(\mathcal E_{\Gamma_1}\otimes\pi_S^*F)\bigl)=\ch_1\bigl(\pi_{\gamma,!}(\mathcal E_{\Gamma_1}\otimes\pi_S^*F)\bigl).
\]
Using the Grothendieck-Riemann-Roch theorem on the projection $\pi_\gamma:S\times\gamma\rightarrow \gamma$ we obtain: 
\begin{equation}\label{E.gamma1}
\ch_1\bigl(\pi_{\gamma,!}(\mathcal E_{\Gamma_1}\otimes\pi_S^*E)\bigl)=
\pi_{\gamma,*}\bigl[\ch(\pi^*_SE\otimes\mathcal E_{\Gamma_1})\td(T_{\pi_{\gamma}})\bigl]_3 
\end{equation}
and 
\begin{equation}\label{F.gamma1}
\ch_1\bigl(\pi_{\gamma,!}(\mathcal E_{\Gamma_1}\otimes\pi_S^*F)\bigl)=
\pi_{\gamma,*}\bigl[\ch(\pi^*_SF\otimes\mathcal E_{\Gamma_1})\td(T_{\pi_{\gamma}})\bigl]_3
\end{equation}
where $T_{\pi_\gamma}:=T_{S\times\gamma}-\pi^*_{\gamma}T_{\gamma}=\pi^*_S T_S\in K(S\times\gamma)$. Given a class \(\alpha\in H^*(S\times\gamma,\ZZ)\), by the notation \([\alpha]_k\) we mean the \(k^{th}\) component of the vector defining \(\alpha\). We start by computing the expression in \eqref{E.gamma1}. We have:
\[
\ch(\pi^*_SE\otimes\mathcal E_{\Gamma_1})\td(T_{\pi_{\gamma}})=\ch(\mathcal E_{\Gamma_1})\pi^*_S\bigl(\ch(E)\td(S)\bigl),
\]

where $\ch(E)\td(S)=\mathfrak v(E)\sqrt{\td(S)}=(1,-h,1)$ because by definition  $\mathfrak v(E)=(1,-h,0)$, and $\sqrt{\td(S)}=(1,0,1)$ since $S$ is a $\K3$ surface.
It follows:
\begin{equation}\label{E.gamma2}
\pi^*_S\bigl(\ch(E)\td(S)\bigl)=(1,-[H\times\gamma],[pt\times \gamma], 0).
\end{equation}

It remains to compute the Chern character of $\mathcal E_{\Gamma_1}=\Gam$; set $G:=\mathcal O_{\gamma\times\gamma}(7(p_0\times\gamma)+\Delta)$. Using the Grothendieck-Riemann-Roch theorem on the inclusion $i:\gamma\times\gamma\hookrightarrow S\times\gamma$, we get:
\[
\ch_{k}\bigl(\mathcal E_{\Gamma_1} \bigl)=i_*\Bigl(\bigl[\ch (G)\td(T_i)\bigl]_{k-1}\Bigl).
\]
where $T_i:=T_{\gamma\times\gamma}-i^* T_{S\times\gamma}\in K(\gamma\times\gamma)$ and $i_{!}G=i_* G=\mathcal E_{\Gamma_1}$ since $i$ is a closed immersion. We have:
\[
\ch(G)=\Bigl(1,c_1(G),\frac{(c_1(G))^2}{2}\Bigl)=\Bigl(1,7[p_0\times\gamma]+\Delta,3\Bigl),
\]
where we used the following intersections:
\begin{itemize}
\item $\Delta^2=\deg(N_{\Delta/\gamma\times\gamma})=\deg(T_{\gamma})=-\deg(\omega_{\gamma})=-8$, since a curve $\gamma\in|2H|$ has $g(\gamma)=5$;
\item $(p_0\times\gamma)^2=0$;
\item $\Delta\cdot(p_0\times\gamma)=1$.
\end{itemize}

To compute $\td(T_i)$ we argue as follows: by the very definition of $T_i$ and by the following short exact sequence
\[
0\rightarrow T_{\gamma\times\gamma} \rightarrow i^*T_{S\times\gamma}\rightarrow N_{\gamma\times\gamma/S\times\gamma}\rightarrow 0
\]
it follows that $T_i=N^{-1}_{\gamma\times\gamma/S\times\gamma}\in K(\gamma\times\gamma)$.
If $\pi_1:\gamma\times\gamma\rightarrow\gamma$ is the projection on the first factor, then $ N^{-1}_{\gamma\times\gamma/S\times\gamma}= \pi_1^*N^{-1}_{\gamma/S}=\pi^*_1\omega^{-1}_{\gamma}$
where the last equality follows from the adjunction formula on the $\K3$ surface $S$. Since $\td(\omega^{-1}_{\gamma})=\bigl(1,\frac{c_{1}(\omega_{\gamma}^{-1})}{2}\bigl)=(1,-4)$, it follows that:
\[
\td(T_i)=(1,-4[pt\times\gamma],0).
\]
Therefore, we get:
\[
\ch(G)\td(T_i)=(1,3[p_0\times\gamma]+\Delta,-1)
\]
and then:
\begin{equation}\label{chGamma}
\ch(\mathcal E_{\Gamma_1})=\bigl(0,[\gamma\times\gamma],3[p_0\times\gamma]+[i_*\Delta],-1\bigl).
\end{equation}
We can finally compute $\bigl[\ch(\pi^*_SE\otimes\mathcal E_{\Gamma_1})\td(T_{\pi_{\gamma}})\bigl]_3$, inserting in \eqref{E.gamma1} what we have computed in \eqref{E.gamma2} and \eqref{chGamma}:
\[
\bigl[\ch(\pi^*_SE\otimes\mathcal E_{\Gamma_1}) \td(T_{\pi_{\gamma}})\bigl]_3=\bigl[\ch(\mathcal E_{\Gamma_1}) \pi^*_S \bigl(\ch(E)\td(S)\bigl)\bigl]_3=-5,
\]
where we used the following intersections:
\begin{itemize}
\item $(\gamma\times\gamma)(pt\times\gamma)=\deg\bigl(N_{\gamma\times\gamma/S\times\gamma}\bigl)|_{pt\times\gamma}=0$
\item $i_*\Delta\cdot(H\times\gamma)=4$
\item $(p_0\times\gamma)(H\times\gamma)=0$.
\end{itemize}
We conclude that $\lambda_v(e)\cdot\Gamma_1=\pi_{\gamma,*}(-5)=-5$.

The computation of $\lambda_c(f)\cdot \Gamma_1$ is similar; we start from \eqref{F.gamma1}, where now $\mathfrak v(F)=(0,0,1)$. We get:
\begin{align*}
\ch(&\pi^*_SF\otimes\mathcal E_{\Gamma_1})\td(T_{\pi_{\gamma}})=\ch(\mathcal E_{\Gamma_1})\pi^*_S\bigl(\ch(F)\td(S)\bigl)=\ch(\mathcal E_{\Gamma_1})\pi_S^*\bigl(\mathfrak v(F)\sqrt{\td(S)}\bigl) \\
&=\ch(\mathcal E_{\Gamma_1})\pi^*_S(0,0,1)=\bigl(0,[\gamma\times\gamma],3[p_0\times\gamma]+[i_*\Delta],-1\bigl)\bigl(0,0,[pt\times\gamma],0\bigl)
\end{align*}
and then:
\[
\bigl[\ch(\pi^*_SF\otimes\mathcal E_{\Gamma_1})\td(T_{\pi_{\gamma}})\bigl]_3=[\gamma\times\gamma][pt\times\gamma]=0,
\]
i.e.\ $\lambda_v(f)\cdot\Gamma_1=0$.

\subsubsection{Intersections $2$ of Proposition \ref{prop.chern}}

Let $\pi_S:S\times \tau\rightarrow S$ and $\pi_\tau:S\times \tau\rightarrow \tau$ be the projections. We need to compute
\[
c_1\bigl(\pi_{\tau,!}(\mathcal E_{\mathcal T_1} \otimes\pi_S^*E)\bigl)=\ch_1\bigl(\pi_{\tau,!}(\mathcal E_{\mathcal T_1}\otimes\pi_S^*E)\bigl)
\]
and
\[
c_1\bigl(\pi_{\tau,!}(\mathcal E_{\mathcal T_1} \otimes\pi_S^*F)\bigl)=\ch_1\bigl(\pi_{\tau,!}(\mathcal E_{\mathcal T_1}\otimes\pi_S^*F)\bigl).
\]
By the Grothendieck-Riemann-Roch theorem on the projection $\pi_\tau$, we get
\[
c_1\bigl(\pi_{\tau,!}(\mathcal E_{\mathcal T_1}\otimes\pi^*_SE)\bigl)=\pi_{\tau,*}\bigl[\ch(\pi^*_SE\otimes\mathcal E_{\mathcal T_1})\td(T_{\pi_\tau})\bigl]_3
\] 
and
\[
c_1\bigl(\pi_{\tau,!}(\mathcal E_{\mathcal T_1}\otimes\pi^*_SF)\bigl)=\pi_{\tau,*}\bigl[\ch(\pi^*_SF\otimes\mathcal E_{\mathcal T_1})\td(T_{\pi_\tau})\bigl]_3,
\] 
where $T_{\pi_\tau}:=T_{S\times \tau}-\pi_\tau^*T_{\tau}=\pi_S^*T_S$ in $K(S\times \tau)$. We start from $\lambda_v(e)\cdot \mathcal T_1$. 

Since $\td(T_{\pi_\tau})=\pi^*_S \td(S)$, as in the previous section we obtain: 
\begin{equation}\label{E.T}
\ch(\pi^*_SE\otimes\mathcal E_{\mathcal T_1})\td(T_{\pi_\tau})=\ch(\mathcal E_{\mathcal T_1})\bigl(1,-[H\times \tau],[pt\times \tau],0\bigl).
\end{equation}

It remains to compute the Chern characters of $\mathcal E_{\mathcal T_1}$.

Set $T:=\mathcal O_{\mathscr C}\bigl(2(q_1\times\tau+...+q_4\times\tau)\bigl)$, so that $\mathcal E_{\mathcal T_1}=j_* T$. Since $j:\mathscr C\hookrightarrow S\times \tau$ is a closed immersion, $j_* T=j_! T$; using the Grothendieck-Riemann-Roch theorem on $j$, we get: 
\[
\ch(\mathcal E_{\mathcal T})_{k}=j_*\bigl([\ch(T)\td(T_j)]_{k-1}\bigl),
\]
where $T_j:=T_{\mathscr C}-j^*T_{S\times\tau}\in K(\mathscr C)$. 

\begin{rem} There exists an isomorphism $\mathscr C\cong \widetilde S:=Bl_{q_1,...,q_8}S$, under which $q_i\times \tau$ gets mapped into the exceptional divisor $E_i$ over the point $q_i$, $i=1,...,8$. From now on, we will use the blow-up notation.
\end{rem}

We have:
\[
\ch(T)=\Bigl(1,c_1(T),\frac{c_1(T)^2}{2}\Bigl)
=(1,2(E_1+E_2+E_3+E_4),-8).
\]

In order to compute $\ch(\mathcal E_{\mathcal T})$ it remains to compute $\td(T_j)=\Bigl(1, \frac{c_1(T_j)}{2}, \frac{c_1(T_j)^2}{12}\Bigl)$. 

By definition of $T_j$, we need to compute: 
\[
\det(T_{\mathscr C})\otimes \det(T_{S\times \tau}|_{\mathscr C})^{-1} =\omega_{\mathscr C}^{-1} \otimes \pi^*_S(\omega_S)|_{\mathscr C}\otimes \pi^*_{\tau}(\omega_\tau)|_{\mathscr C}.
\]
We will use the following facts:
\begin{itemize}
\item $\omega_S=0$ since $S$ is $\K3$, and then $\pi^*_S(\omega_S)|_{\mathscr C}=0$;
\item since $\tau\cong \mathbb P^1$, $\omega_\tau=\mathcal O_{\tau}(-2)$. Furthermore, for $t\in \tau$, one has $\pi_{\tau}^{-1}(t)|_{\mathscr C}\cong C_{t}\times t$, with $C_t$ the curve corresponding to $t$, and then $\pi^*_{\tau}(\omega_\tau)|_{\mathscr C}=\mathcal O_{\mathscr C}(-2(C_t\times t))$;
\item since $\mathscr C\cong \tilde S$, if $\pi:\mathscr C\rightarrow S$ is the blow-up morphism then $\omega_{\mathscr C}^{-1}\cong\pi^*\omega_S^{-1}-\sum_{i=1}^8E_i=-\sum_{i=1}^8E_i$.
\end{itemize}

We conclude:
\[
\td(T_j)=\Bigl(1, -[C_t\times t]-\sum_{i=1}^8\frac{[E_i]}{2}, 2\Bigl),
\]
where we used that $(C_t\times t)^2=0$ and $(C_t\times t)\cdot E_i=1$, $\forall i=1,...,8$.

In conclusion, we have:
\[
\ch(T)\td(T_j)=\Bigl(1,2[E_1+E_2+E_3+E_4]-[C_t\times t]-\sum_{i=1}^8\frac{[E_i]}{2}, -10 \Bigl)
\]
and we get:
\begin{equation}\label{chT}
\ch(\mathcal E_{\mathcal T_1})=\Bigl(0,[\mathscr C], 2[E_1+E_2+E_3+E_4]-[C_t\times t]-\sum_{i=1}^8\frac{[E_i]}{2},-10\Bigl).
\end{equation}
Using the expression \eqref{chT} just computed, we obtain that the expression \eqref{E.T} equals to:
\begin{align*}
[\mathscr C][pt\times \tau]-&2[\sum_{i=1}^4E_i][H\times \tau]+[C_t\times t][H\times \tau] +\sum_{i=1}^8\frac{[E_i]}{2}[H\times \tau]-10.
\end{align*}
Notice that:
\begin{itemize}
\item $(\mathscr C)(pt\times \tau)=|\{(pt,t)\in S\times \tau\ |\ pt\in C_t\}|$; we can assume $pt=p\notin\{q_1,...,q_8\}$, and such a $p$ determines an unique $C_p\in \tau$. It follows $(\mathscr C)(pt\times \tau)=1$;
\item $E_i\cdot(H\times\tau)=(p_i\times \tau)(H\times \tau)=0$ $\forall i=1,...,8$;
\item $(C_t\times t)(H\times \tau)=C_t\cdot H=4$;
\end{itemize}
It follows that:
\[
\bigl[\ch(\pi^*_SE\otimes\mathcal E_{\mathcal T_1})\td(T_{\pi_\tau})\bigl]_3=-5=\lambda_v(e)\cdot \mathcal T_1.
\]
We pass to $\lambda_v(f)\cdot \mathcal T_1$. We need to compute:
\[
c_1\bigl(\pi_{\tau,!}(\mathcal E_{\mathcal T_1}\otimes\pi^*_SF)\bigl)=j_*\bigl[\ch(\pi^*_SF\otimes\mathcal E_{\mathcal T_1})\td(T_{\pi_\tau})\bigl]_3.
\]
where now
\[
\pi^*_S(\ch(F)\td(S))=\pi^*_S(0,0,1)=(0,0,[pt\times \tau],0).
\]
Using \eqref{chT} we get:
\[
\lambda_v(f)\cdot\mathcal T_1=j_*\Bigl(\bigl[\ch(\pi^*_SF\otimes\mathcal E_{\mathcal T_1})\td(T_{\pi_\tau})\bigl]_3\Bigl)=j_*\Bigl([\mathscr C][pt\times \tau]\Bigl)=1.
\]
\\
This completes the proof of Proposition \ref{prop.chern}.

\subsection{The $v_2$-case}\label{subsection:v2 intersection curves} We start fixing a $\ZZ$-basis for $v_2^\perp$. Observe that:
\[
(v_2^\perp)^{1,1}=\{(2a,ah,b)|\ a,b\in\ZZ\}.
\] 
We choose the following generators:
\begin{equation}\label{eq.e.f.in.M2}
e:=(2,h,0),\ f:=(0,0,1).
\end{equation}
Let $\Gamma_2,\cT_2\subset\M_{v_2}$ be the curves introduced in Definition \ref{def.curves2}.

\begin{prop}\label{prop.inters.lambda4} Let \(e,f\in (v_2^\perp)^{1,1}\) be as in \eqref{eq.e.f.in.M2} and consider the isometry \(\lambda_{v_2}:(v_2^\perp)^{1,1}\tilde{\rightarrow}\Pic(\M_{v_2})\) introduced in Corollary \ref{cor.isom.hodge}. One has the following intersections:
\begin{enumerate}
\item $\lambda_{v_2}(e)\cdot \Gamma_2=-10$, $\lambda_{v_2}(f)\cdot \Gamma_2=0$.
\item $\lambda_{v_2}(e)\cdot \mathcal T_2=-8$, $\lambda_{v_2}(f)\cdot \mathcal T_2=1$.
\end{enumerate}
\end{prop}
\begin{proof}
The proof goes exactly as the proof of Proposition \ref{prop.chern}, then we will only sketch the necessary computations. Let us fix $[E]$, $[F]\in K(S)$ with $\mathfrak v(E)=e^\vee=(2,-h,0)$ and $\mathfrak v(F)=f^\vee=(0,0,1)$. We will use notation analogous to the ones in the proof of Proposition \ref{prop.chern}.
\begin{itemize}
\item We need to compute
\begin{equation}\label{eq.lamda4}
\lambda_{v_2}(e)\cdot\Gamma_2=\bigl[\ch(\pi^*_S E\otimes\mathcal E_{\Gamma_2})\td(T_{\pi_\gamma})\bigl]_3
\end{equation}
where
\begin{align*}
\ch(\pi^*_S E\otimes\mathcal E_{\Gamma_2})\td(T_{\pi_\gamma})&=\pi^*_S\bigl(\ch(E)\td(S)\bigl)\ch(\mathcal E_{\Gamma_2}) \\
&=\pi^*_S(2,-h,2)\ch(\mathcal E_{\Gamma_2}) \\
&=(2, -[H\times\gamma],2[pt\times\gamma],0)\ch(\mathcal E_{\Gamma_2}).
\end{align*}
Since $\ch(\mathcal E_{\Gamma_2})=i_*\bigl(\ch(G)\td(T_i)\bigl)$, with $G:=\mathcal O_{\gamma\times\gamma}(5p_0\times\gamma+\Delta)$, we have:
\[\ch(G)=(1,5[pt\times\gamma]+\Delta,1)\]
and
\[\td(T_i)=(1,-4[pt\times\gamma],0)\].
Then:
\[
\ch(\mathcal E_{\Gamma_2})=(0,[\gamma\times\gamma],[pt\times\gamma]+[i_*\Delta],-3).
\]
Combining everything in equation \eqref{eq.lamda4} one gets $\lambda_{v_2}(e)\cdot\Gamma_2=-10$.
\item $\lambda_{v_2}(f)\cdot\Gamma_2=0$ follows from 
\[
\ch(\pi^*_S F\otimes\mathcal E_{\Gamma_2})\td(T_{\pi_\gamma})=(0,0,[pt\times\gamma],0)\ch(\mathcal E_{\Gamma_2}).
\]
\item We now need to compute $\lambda_{v_1}(e)\cdot\cT_2=\bigl[\ch(\mathcal E_{\cT_2})(2,-[H\times\tau],2[pt\times\tau],0)\bigl]_3$, where $\ch(\mathcal E_{\cT_2})=j_*\bigl(\ch(T)\td(T_j)\bigl)$, with $T:=\cO_{\mathscr C}(3q_1\times\tau+...+q_4\times\tau)$. One has
\[\ch(T)=(1,[3E_1+...+E_4],-6)\]
and
\[\td(T_j)=\Bigl(1,-[C_t\times t]-\bigl[\sum_{i=1}^8\frac{E_i}{2}	\bigl],2\Bigl).\]
Therefore 
\[
\ch(\cE_{\cT_2})=\Bigl(1,[\mathscr C],[3E_1+...+E_4]-\Bigl[\sum_{i=1}^8\frac{E_i}{2}\Bigl]-[C_t\times t],-7\Bigl)
\]
which gives $\lambda_{v_2}(e)\cdot\cT_2=-8$.
\item $\lambda_{v_2}(f)\cdot\cT_2=1$ follows from 
\[
\ch(\pi^*_S F\otimes\mathcal E_{\cT_2})\td(T_{\pi_\tau})=(0,0,[pt\times\tau],0)\ch(\mathcal E_{\cT_2}).
\]
\end{itemize}
\end{proof}

This concludes point (2) of Strategy \ref{strategy} for the moduli spaces $\M_{v_1}$ and $\M_{v_2}$.


\section{The classes of the divisors}\label{section:class of the divisors}

In this section we will deal with point (3) and (4) of Strategy \ref{strategy}, applied to the uniruled divisors $D_1$ and $D_2$ defined in Section \ref{section:uniruled divisors}. Notice that manifolds $\M_{v_1}$ and $\M_{v_2}$, where the divisors $D_1$ and $D_2$ live, are not smooth; it follows that a Weil divisor can happen to be non Cartier. As consequence, a preliminary step in the computation in Strategy \ref{strategy}.(3) will actually be to check that $D_1$ and $D_2$ are Cartier divisors.

This section in the technical core of this article, and it is the longest one. For this reason, we briefly sketch here the structure of the section, to guide the reader through the computations that will come. 

The section is divided in two subsections: in the first one we will deal with the divisor $D_1$ and in the second one with the divisor $D_2$. In both the subsections, we proceed as follows: we check that the divisor $D_i$ is Cartier, then we compute the intersections $D_i\cdot \Gamma_i$ and $D_i\cdot\cT_i$. Once done, following that Strategy \ref{strategy} we compute the class of $D_i^*$ in $H^2(\widetilde\M_{v_i},\ZZ)$ and the square $q_{10}(D_i)=q_{10}(D_i^*)$. As observed in Remark \ref{rem.pullback.strict.trans}, we prefer to work with the strict transform $\widetilde D_i\subseteq\widetilde\M_{v_i}$ instead of $D_i^*$; for this reason, once we compute the class of $D_i^*$ in $H^2(\widetilde\M_{v_i},\ZZ)$, we pass to the divisor $\widetilde D_i$ and we compute its class and square; we conclude computing the class and the square of the divisor \(B_i=\widetilde D_i+\widetilde\Sigma_i\).

\subsection{The class of $\widetilde D_1$ and of \(B_1\)}\label{subsection:class D1} In this section we work with the uniruled divisor $D_1\subset\M_{v_1}$ introduced in Definition \ref{def.D}.

\begin{rem} The divisor $D_1$ is a Cartier divisor. In order to check it, we just need to apply Theorem \ref{thm.factorial} to our case: $v_1=(0,2h,4)$, then $w=(0,h,2)$; an element $\gamma\in (H^*(S))^{1,1}$ is a vector of the form $(a,bh,c)$, then 
\[
\gamma\cdot (0,h,2)=2b-2a\in 2\ZZ
\] 
for any $\gamma\in (H^*(S))^{1,1}$. It follows that $\M_v$ is locally factorial and the divisor $D_1$ is a Carter divisor.
\end{rem}

As consequence, we can develop point (3) of Strategy \ref{strategy} and  compute the intersection of $D_1$ with the curves $\Gamma_1$ and $\cT_1$ introduced in Definition \ref{def.curves1}.

\begin{rem}\label{rem.theta} Let $C$ be a smooth curve.  $J^k(C)$ is a torsor on the abelian variety $J^0(C)$. If $g$ is the genus of $C$, then the subtorsor of $J^{g-1}(C)$ consisting of effective divisors up to linear equivalence is a theta divisor in $J^{g-1}(C)$ (this can be found for example in \cite{birkenhakelange}). It follows that, for $r\in J^{k-g+1}(C)$, 
\[
\theta_r:=\{r+\alpha|\ \alpha\ \textrm{is effective}\}
\] 
is a theta divisor in $J^k(C)$, since it is a translate of the previous one. This will be used in the following Proposition.
\end{rem}

\begin{prop}\label{prop.D.Gamma}
Let $D_1$ be the divisor introduced in Definition \ref{def.D} and $\Gamma_1$ as in Definition \ref{def.curves1}. Then $D_1\cdot \Gamma_1=20$. 
\end{prop}

\begin{proof}
Let $p:\M_{v_1} \rightarrow |2H|$ the projection defined in Subsection \ref{subsection:Lagrangian}; $\Gamma_1$ is supported inside the fiber $p^{-1}(\gamma)\cong J^{8}(\gamma)$. It follows that we need to compute the intersection of $D_1$ and $\Gamma_1$ only inside $J^{8}(\gamma)$; for simplicity, we set $D_\gamma:=D_1|_{J^{8}(\gamma)}$.

If $AJ:\gamma\rightarrow J^8(\gamma)$ is the Abel-Jacobi map translated by $7p_0$, then 
\[
[\Gamma_1]=AJ_*[\gamma]=\Bigl[\frac{\theta^4}{4!}\Bigl]\in H^2(J^8(\gamma),\ZZ),
\] 
where $[\theta]$ is the class of a theta divisor on $J^8(\gamma)$.
If we denote by $\{r_1,r_2,r_3,r_4\}=\rho_0\cap \gamma$, then $D_\gamma=\sum_{i=1}^4D_{r_i}$, where $D_{r_i}$ is the component of $D_\gamma$ obtained by fixing the point $r_i$ in the intersection $\rho_0\cap\gamma$; by Remark \ref{rem.theta}, $[D_{r_i}]=[\theta]$.

Using the Poincar\'e formula, it follows that:
\[
\Gamma_1\cdot D_1=4\Bigl[\frac{\theta^4}{4!}\Bigl]\cdot[\theta]=4\Bigl[\frac{\theta^5}{4!}\Bigl]=20.
\]
\end{proof}

We pass now to the intersection $D_1\cdot \cT_1$. Before going to the computation we state and prove here two technical results, that we will use in computing the intersection $D_1\cdot \cT_1$ in Proposition \ref{prop.D.T}.

\begin{lem}\label{lemma.D.T.liscia} Let $D_1$ be the divisor introduced in Definition \ref{def.D}, and let $\cT_1$ be as in Definition \ref{def.curves1}. For any $[L]\in\cT_1\cap D_1$ with \(\Supp(L)=C\) smooth curve, it holds $h^0(C,L(-4r))\le 1$ for any $r\in C\cap\rho_0$. As consequence, if $\Supp(L)$ is smooth then there exists an open neighborhood $V\subseteq D_1$ containing $[L]$ such that $V$ is the union of finitely many $V_i$, which are smooth and such that $\cap_i V_i=[L]$.
\end{lem}
\begin{proof}
We start proving the second part of the statement, assuming the first one. Take \([L]\in\cT_1\cap D_1\) and \(r\in \rho_0\cap C\) such that \(h^0(C,L(-4r))\ge 1\), where \(C=\Supp(L)\) as in the statement; such \(r\) exists because \([L]\in D_1\). When \(C\) is smooth, by the first part of the statement we have \(h^0(C,L(-4r))=1\), hence $L$ can be uniquely written as a sheaf of the form $i_*\cO_C(\alpha+4r_i)$, where $i:C\hookrightarrow S$ is the inclusion of the curve in the \(\K3\) surface and $\alpha$ is an effective divisor of degree 4 on $C$. It follows that the map defined in Lemma \ref{lem.D.uniruled}:
\begin{align*}
\phi:\rho_0\times \Sym^4(S)&\dashrightarrow D_1 \\
\xi:=(r,[p_1,p_2,p_3,p_4])&\mapsto i_*\cO_{C_\xi}(4r+p_1+p_2+p_3+p_4)
\end{align*}
is injective on (the unique) preimage of $[L]$, and indeed it is injective on an open neighborhood $U$ containing $\phi^{-1}[L]$; up to shrinking it, we can assume that $U$ is smooth in $\rho_0\times \Sym^4(S)$, and that $\phi$ is defined and injective on $U$. Thus, by the Zariski's main theorem, the image of $\phi(U)=:V$ is smooth. We can repeat this construction for any $r_i\in C\cap\rho_0$ such that $h^0(L(-4r_i))=1$ when $C$ is smooth. We get smooth open subsets $V_i\subset D_1$ containing $[L]$; note that $C\cap\rho_0$ consists of at most four points, since $C\cdot\rho_0=4$. The union $V$ of such $V_i$ is an open neighborhood of $[L]$ in $D_1$, and the intersection of the open subsets $V_i$  only consists of the point $[L]$, as in the statement.

We pass to the first part of the statement. We will prove that the condition \(h^0(C,L(-4r))\le 1\) in the statement holds true for a generic choice of the pencil \(\tau\) defining \(\cT_1\); notice that the pencil and the curve \(\cT_1\) is uniquely individuated by the choice of four points on the \(\K3\) surface that are not conjugated via the involution \(\iota\) on \(S\) (cf. Remark \ref{rem.conics}), meaning here that \(\iota(q_i)\neq q_j\) for any \(i,j=1,2,3,4\).

Pencils in the linear system $|2H|\cong \PP^5$ are parametrized by the Grassmannian $\Gr(2,6)$. We call:
\[
I\subseteq \M_{(0,2h,0)}\times\Gr(2,6)
\]
the incidence variety consisting of pairs \(([L],\tau)\in \M_{(0,2h,0)}\times\Gr(2,6)\) such that: \(L|_{\Supp(L)}\cong\cO_{\Supp(L)}(2(q_1+q_2+q_3+q_4)-4r)\) for some \(r\in \Supp(L)\cap\rho_0\), \(h^0(L)\ge2\) and \(Bs(\tau)=\{q_i,\iota(q_i)\}_{i=1}^4\). Furthermore, we call \(p:I\rightarrow\M_{(0,2h,0)}\) and \(q:I\to\Gr(2,6)\) the projections. We want to show that the projection \(q\) is not dominant. 

Assume that \(q:I\to\Gr(2,6)\) is dominant. Since \(\dim(\Gr(2,6))=8\), it follows that the incidence variety \(I\) has dimension at least 8. On the other hand, the locus of \(\M_{(0,2h,0)}\) consisting of sheaves with at least two sections has dimension 7: the moduli space \(\M_{(0,2h,0)}\) contains the relative Jacobian \(\cJ^4_{|2H|^{sm}}\) (see Subsection \ref{subsection:Lagrangian}), and given a smooth curve \(C\in |2H|\) the Brill-Noether locus \(W^1_4(C)\) has dimension \(2\), because \(C\) is hyperelliptic (see Martens' theorem \cite[Theorem 5.1]{ACGH}). We conclude that the generic fiber of \(p:I\to\M_{(0,2h,0)}\) has positive dimension. Two distinct points in \(I\) are in the same generic fiber of \(p:I\to\M_{(0,2h,0)}\) if there exists a (smooth) curve \(C\sim 2H\), points \(q_1,...,q_4,p_1,...,p_4\in C\) such that the \(q_i\)'s and the \(p_i\)'s are not conjugated via the involution \(\iota\), and points \(r,s\in C\cap \rho_0\) such that \[2(q_1+q_2+q_3+q_4)-4r\sim 2(p_1+p_2+p_3+p_4)-4s\] as divisors on \(C\). This holds true if and only if \[q_1+q_2+q_3+q_4\sim p_1+p_2+p_3+p_4-2(s-r)+\xi\] for a 2-torsion cycle \(\xi\in J^0(C)\). Since by assumption the fiber of \(p\) has positive dimension and \(C\cap\rho_0\) and the set of 2-torsion points in \(J^0(C)\) are finite sets, we can assume that the relation above holds true for \(r=s\) and \(\xi=0\), hence
\[
q_1+q_2+q_3+q_4\sim p_1+p_2+p_3+p_4.
\]
Since the two points on the generic fiber of \(p\) we are considering are distinct, it follows that \(h^0(C,\cO_C(q_1+q_2+q_3+q_4))\ge 2\), hence by the geometric version of the Riemann-Roch theorem on a hyperelliptic curve presented in \cite[Section 1]{ACGH} we have \(q_1+q_2+q_3+q_4=\alpha+a_1+a_2\), for some \(a_1,a_2\in C\) and \(\alpha\in\mathfrak g^1_2\) on \(C\). We conclude that there exists \(i,j\) such that \(q_i+q_j\in \mathfrak g^1_2\), which contradicts the assumption that no \(q_i\)'s are conjugated via the involution \(\iota\). We conclude that \(q:I\to\Gr(2,6)\) is not dominant, hence the first part of the statement holds true, and the Lemma is proved.
\end{proof}

\begin{lem}\label{lem.checks.D.T} Up to changing the very general \(\K3\) surface \(S\) and the generic pencil \(\tau\subseteq|2H|\) with base points \(\{q_1,...,q_4,\iota(q_1),...,\iota(q_4)\}\), the following conditions hold true.
\begin{enumerate}
\item The curve \(\cT_1\) defined from the pencil \(\tau\) as in Definition \ref{def.curves1} does not intersect the divisor \(D_1\) in points corresponding to sheaves supported on a reducible curve parametrized by \(\tau\).
\item Take any of the three reducible curves parametrized by  the pencil \(\tau\); we call it \(C:=C_1\cup C_2\). Assume that \(q_1,q_2\in C_1\) and \(q_3,q_4\in C_2\); for any \(r\in C_1\cap \rho_0\) and \(s\in C_2\cap \rho_0\), one has:
\[
h^0(C,\cO_C(2(q_1+q_2)+q_3+q_4-(\iota(q_3)+\iota(q_4))-4r+\alpha+\beta))=0,
\]
\[
h^0(C,\cO_C(2(q_1+q_2)+q_3+q_4-(\iota(q_1)+\iota(q_2))-4s+\alpha+\beta))=0
\]  
for any \(\alpha\in\mathfrak g^1_2\) on \(C_1\) and \(\beta\in \mathfrak g^1_2\) on \(C_2\).
\end{enumerate}
\end{lem}
\begin{proof}
Fix a singular curve \(C:=C_1\cup C_2\) parametrized by the pencil \(\tau\); we call \(p:\M_{v_1}\to |2H|\) the projection introduced in Subsection \ref{subsection:Lagrangian}. Regarding \((1)\), observe that \(D_1\) intersect \(p^{-1}(C)\) in a locus of codimension one in \(p^{-1}(C)\); it follows that the statement in \((1)\) is proved for the sheaf in \(\cT_1\) supported on \(C\) if we show that \(L:=\cO_C(2(q_1+q_2+q_3+q_4)-4r)\) is generic in \(p^{-1}(C)\), for any \(r\in C\cap\rho_0\). Furthermore, the sheaves in part \((2)\) are proved to be generic once \(L\) is proved  to be generic; since they have degree two on both the irreducible components of \(C\), if generic they have no global sections, hence also \((2)\) will be proved on the curve \(C\).

The first observation is that, up to changing the pencil \(\tau\) and the very general \(\K3\) surface where \(C\) lives, the points \(q_1,...,q_4,r\) defining \(L''\) in \eqref{eq.L''} can vary freely. Indeed, the \(\K3\) surface \(S\) is determined by the choice of a generic sextic in \(\PP^2\); sextics in \(\PP^2\) are parametrized by a projective space of dimension 27. We call \(\gamma=\gamma_1\cup\gamma_2\) the singular conic of \(\PP^2\) such that \(C\) is the pullback of \(\gamma\) on the \(\K3\) surface via the double covering ramified along the sextic. Fixed \(\gamma\), we want the sextic in \(\PP^2\) to intersect \(\gamma\) in 12 distinct points, six on \(\gamma_1\) and six on \(\gamma_2\); once chosen such 12 points on \(\gamma\), sextics in \(\PP^2\) intersecting transversally \(\gamma\) in these points are parametrized by a projective space of dimension 5. Finally, for any choice of \(x\in \gamma_1\) and \(y\in \gamma_2\) with \(x\neq y\) we can consider the line passing through them; this defines a rational curve on the \(\K3\) surface when the sextic is bitangent to it, and since each tangency is a linear condition on the space of the sextics, fixed \(x\) and \(y\) the space of sextics in \(\PP^2\) intersecting \(\gamma\) in the 12 given points and bitangent to the line given by \(x\) and \(y\) are paramentrized by a 3-dimensional projective space.

Summing up, we find that fixed the curve \(\gamma\subseteq\PP^2\) for every choice of \(x,y\in \gamma\) as above there exists a \(\K3\) surface that is a double covering of \(\PP^2\) ramified along a sextic, which intersects \(\gamma\) transversally in 12 points; the curve \(C\) is the pullback of \(\gamma\) via the double covering. Furthermore, the choice of the pencil \(\tau\) is given by any generic choice of 2 points on \(\gamma_1\) and 2 points on \(\gamma_2\), whose preimages are the points \(q_1,...,q_4,\iota(q_1),...,\iota(q_4)\); here, \(\iota\) is the involution on \(S\) given by the double covering. We conclude that, up to changing the very general \(\K3\) surface \(S\) and the pencil \(\tau\), the points \(q_1,...,q_4,r\) defining \(L\) can vary freely.

Varying freely the points \(q_1,...,q_4.r\), the line bundle \(L\) is generic among the line bundles of degree 4 on \(C\). Indeed, this is equivalent for the line bundle \(\cO_C(q_1+q_2+q_3+q_4-2r)\) to be generic among the line bundles of degree 4. This is obtained as combination of the fact that, fixed a point \(p_0\in C\), the map \(\Sym^5(C)\to J^0(C)\), \([p_1,...,p_5]\mapsto \sum_{i=1}^5 (p_i-p_0)\) is dominant and generically injective, and the fact that the map \(J^0(C)\to J^0(C)\), \(x\mapsto 2x\) is dominant. We conclude that the statements in \((1)\) and \((2)\) hold true on the curve \(C=C_1\cup C_2\).

We have proved the statement  for one of the three singular curves parametrized by the pencil \(\tau\), and we want to show that it holds true also for the remaining two singular curves, always for a generic choice of the pencil \(\tau\). The Grassmannian \(\Gr(2,6)\) parametrizes the pencils in \(|2H|\cong\PP^2\); we call \(\cG\) the tautological bundle on \(\Gr(2,6)\), and we consider its projectivization \(\pi:\PP(\cG)\to\Gr(2,6)\), whose points are written as pairs \((\tau,t)\), with \(\tau\) pencil in \(\PP^5\) and \(t\in\tau\). Also, let \(R\subseteq |2H|\) be the subset of reducible curves having two smooth irreducible components intersecting in two distinct points. Given the inclusion \(j:R\hookrightarrow |2H|\) and the projection \(p:\PP(\cG)\rightarrow |2H|\), \((\tau,t)\mapsto t\), we consider the pullback: 
\[
\begin{tikzcd}
I:=\PP(\cG)\times_{|2H|} R \arrow{r}{\hat p}\arrow{d}{\hat j} & R \arrow{d}{j} \\
\PP(\cG) \arrow{r}{f} & \textbar 2H\textbar 
\end{tikzcd}
\]
Observe that \(I\cong\{((\tau,t)\in \Gr(2,6)\times R\ |\ t\in \tau\}\). Since \(\PP(\cG)\) and \(|2H|\) are smooth varieties and \(f\) has equidimensional fibers, by miracle flatness the morphism \(f\) is flat. It follows that \(\hat f\) is flat as well, hence any irreducible component of \(I\) dominates \(R\); but the fibers of \(f\) are irreducible, hence only an irreducible component of \(I\) can dominate \(R\). We conclude that \(I\) is irreducible, hence there exists no rational section of the generically 3:1 dominant map  \(\pi':I\to \Gr(2,6)\) induced by the projection \(\pi:\PP(\cG)\to \Gr(2,6)\).

Take a generic pencil \(\tau\in \Gr(2,6)\). We have proved that there exists a reducible curve \(C\) of the pencil such that the statement in \((1)\) and the statement in \((2)\) hold true. Assume that the statement in \((1)\) (resp. in \((2)\)) does not hold true for the remaining two reducible curves of \(\tau\); this would give a rational section \(\Gr(2,6)\dashrightarrow I\), associating to the generic pencil \(\tau\) the only reducible curve such that the statement in \((1)\) (resp. in \((2)\)) holds true. Similarly, one construct a rational section \(\Gr(2,6)\dashrightarrow I\) if the statement in \((1)\) (resp. in \((2)\)) holds true for only one of the remaining two reducible curves, associating to the pencil the only reducible curve such that the statement in \((1)\) (resp. in \((2)\)) does not hold true. We conclude that the statement in \((1)\) and \((2)\) hold true on all the three reducible curves of the generic pencil \(\tau\), and the Lemma is proved.
\end{proof}

\begin{prop}\label{prop.D.T} Let $D_1$ be the divisor introduced in Definition \ref{def.D} and $\mathcal T_1$ as in Definition \ref{def.curves1}. Then $D_1\cdot \mathcal T_1=72$.
\end{prop}  
\begin{proof}
We divide the proof in two parts: in the first one, we will give a reformulation of the intersection, in order to arrive to a formula that is easier to compute; in the second one, we will proceed with the actual computation.
\\

\underline{\textit{Part 1: Reformulation of the intersection.}} First, notice that $\cT_1\subseteq\M_{v_1}^s$: on the three singular curves in $\tau$, which are union of two smooth irreducible components, the corresponding sheaf in $\cT_1$ is such that its restriction on both components has degree 4, which is a stable configuration. It follows that the family $\cE_{\cT_1}$ of Definition \ref{def.D} defines an inclusion $j:\tau\hookrightarrow \M_{v_1}^s$. We need to compute $D_1\cdot \cT_1=\deg(j^* D_1)$.

We start from $D'_1\cdot \cT_1=\deg(j^* D'_1)$. We will use the notation of Remark \ref{rem.D.schematic}, we can work on the stable locus $\M_{v_1}^s$ since it contains $D'_1\cap\cT_1$. Given the projection $p_{\M^s_{v_1}}:I_{\rho\times\M^s_{v_1}}\rightarrow\M^s_{v_1}$, consider the induced commutative diagram:
\[
\begin{tikzcd}
I_{\rho\times\M^s_{v_1}} \arrow{r}{p_{\M^s_{v_1}}}  & \M^s_{v_1} \\
I_{\rho\times\tau} \arrow{r}{p_\tau}\arrow[hook, u, "\hat j"] & \tau \arrow[hook, u, "j"]
\end{tikzcd}
\]
Observe that, by construction, $I_{\rho\times\tau}$ is the following incidence variety:
\[
I_{\rho\times\tau}=\{(r,t)\in \rho\times\tau|\ \nu(r)\in C_t\}
\]
where $C_t\subset S$ is the curve corresponding to $t\in \tau$; as usual, $\nu:\rho\rightarrow \rho_0$ is the normalization of the rational curve $\rho\in|H|$. Also, the incidence variety $I_{\rho\times\tau}$ is isomorphic to the curve $\rho$: an isomorphism is given by 
\begin{align*}
\rho&\rightarrow I_{\rho\times\tau} \\
r&\mapsto(r,t_r)
\end{align*}
where $t_r\in \tau$ is the point corresponding to the unique curve in $\tau$ passing through $\nu(r)\in \rho_0$; this curve is well defined for every $r\in\rho$, because the base points $q_1,...,q_8$ of $\tau$ do not belong to the rational curve $\rho_0$. In particular, $I_{\rho\times\tau}$ is a smooth variety. 

By definition of $D_1'$, following again the notation introduced in Remark \ref{rem.D.schematic} we have: 
\[
j^*\bigl(p_{\M^s_{v_1}}(D_1'')\bigl)=j^*D_1'.
\]
Observe that $j$ is a closed immersion, hence proper, and $p_\tau$ is a finite surjective morphism of non singular manifolds, hence flat. Under these hypotheses pushforward and pullback commute along cartesian squares, and we get:
\begin{align*}
D'_1\cdot\cT_1&=\deg(j^* D'_1)=\deg\bigl(j^*(p_{\M^s_{v_1}}(D_1''))\bigl)=\deg\bigl(p_\tau(\hat j^*D_1'')\bigl)=\deg(\hat j^*D_1'').
\end{align*}

We pass to $D_1\cdot\cT_1$. In Remark \ref{rem.D1eD1'} we have noticed that that $D_1\subsetneq D'_1$ and that the divisors coincide once restricted to $\cJ^8_{|2H|^{sm}}\subseteq\M_{v_1}$. Regarding the three sheaves of \(\cT_1\) supported on reducible curves, by Lemma \ref{lem.checks.D.T}.(1) they do not belong to \(D_1\).
We conclude that:
\[
D_1\cdot\cT_1=(D'_1\cdot\cT_1)|_{\cJ^8_{|2H|^{sm}}}=\deg(\hat j^*D_1'')|_{I_{\rho\times\tau}\smallsetminus A},
\]
where $A\subseteq I_{\rho\times\tau}$ is the set of points that are mapped via $p_\tau$ on a point in $\tau$ representing one of the singular curves $C_i\cup C_i'$, $i=1,2,3$. We call $U:=I_{\rho\times\tau}\smallsetminus A$.

We proceed with a reformulation of this last intersection.  In order to compute $\deg(\hat j^*D_1'')|_U$, we construct one more commutative diagram. Consider:
\begin{itemize}
\item the projection $p_{\rho\times\M^s_{v_1}}:I_{\rho\times S\times\M^s_{v_1}}\rightarrow I_{\rho\times\M^s_{v_1}}$;
\item the incidence variety $\mathscr C\subset \tau\times S$ defining $\cT_1$, as introduced right before Definition \ref{def.curves1}, and the projection $q:\mathscr C\rightarrow\tau$;
\item the incidence variety: $I_{\rho\times\tau\times S}=\{(r,t,x)\in \rho\times\tau\times S|\ \nu(r),x\in C_t\}$, where \(C_t\subseteq S\) is the curve given by \(t\in \tau\);
\item the natural projection $p_{\mathscr C}:I_{\rho\times\tau\times S}\rightarrow \mathscr C$,
which is a covering of degree 4.
\end{itemize} 
We get the following commutative diagram:
\begin{equation}\label{diag.incidental}
\begin{tikzcd}
I_{\rho\times S\times\M^s_{v_1}} \arrow{r}{p_{\rho\times\M^s_{v_1}}}  & I_{\rho\times\M^s_{v_1}} \\
I_{\rho\times\tau\times S} \arrow{r}{p_{\rho\times\tau}}\arrow[hook, u, "\tilde j"]\arrow{d}{p_{\mathscr C}} & I_{\rho\times\tau} \arrow[hook, u, "\hat j"]\arrow{d}{p_\tau} \\
\mathscr C \arrow{r}{q} & \tau
\end{tikzcd}
\end{equation}

By definition: 
\[
D_1''=\Supp\bigl(R^1p_{\rho\times\M^s_{v_1},*}\hat\cF\otimes\cI_{R}^{\otimes 4}\bigl)
\] 
where $\hat\cF$ is the pullback on $\rho\times S\times\M^s_{v_1}$ of the quasi-universal family $\cF\in\Coh(S\times\M^s_{v_1})$ that we fixed to define the schematic structure of $D_1$ in Remark \ref{rem.D.schematic}. We get:
\begin{align*}
\deg(\hat j^*D_1'')|_U=\deg(\hat j^*\Supp\bigl(R^1p_{\rho\times\M_{v_1},*}\hat \cF\otimes\cI_{R}^{\otimes 4}\bigl))|_U=\deg\bigl(\hat j^*R^1p_{\rho\times\M_{v_1},*}\hat \cF\otimes\cI_{R}^{\otimes 4}\bigl)|_U
\end{align*}
where the last equality follows from Lemma \ref{lemma.D.T.liscia}: $\hat j^*R^1p_{\rho\times\M_{v_1},*}\hat \cF\otimes\cI_{R}^{\otimes 4}|_U$ is a sheaf on $U\subseteq I_{\rho\times\tau}$ with fibers of rank 0 or 1, and its degree equals the length of its support; notice that Lemma \ref{lemma.D.T.liscia} regards the divisor $D_1$, but we can apply it to $D_1'|_{\cJ^8_{|2H|^{sm}}}$ since $D_1|_{\cJ^8_{|2H|^{sm}}}=D_1'|_{\cJ^8_{|2H|^{sm}}}$. From the commutative diagram in \eqref{diag.incidental} it follows that:
\begin{align*}
\deg(\hat j^*D_1'')|_U&=\deg\bigl(\hat j^*R^1p_{\rho\times\M_{v_1},*}\hat \cF\otimes\cI_{R}^{\otimes 4}\bigl)|_U=\deg\bigl(R^1 p_{\rho\times\tau,*}\tilde j^*(\hat\cF\otimes\cI_{R}^{\otimes 4})\bigl)|_U \\
&=\deg\Bigl(R^1p_{\rho\times\tau,*}\bigl(p_{\mathscr C}^*\cE_\cT\otimes\widetilde j^*\cI_{R}^{\otimes 4}\bigl)\Bigl)|_U.
\end{align*}
where the second equality follows from the base change and the fact that we are working with the top cohomology, and the last one by the fact that $\hat\cF$ is the pull-back of a quasi-universal family, then the sheaves $\widetilde j^*\hat\cF$ and $p^*_{\mathscr C}\cE_\cT$ on $I_{\rho\times\tau\times S}$ are isomorphic up to multiplication with a line bundle, which does not affect the degree computation we are doing since $I_{\rho\times\tau}\cong\rho$ is smooth. 

Summarizing, and calling $\hat R:=\tilde j^*R\subset I_{\rho\times\tau\times S}$:
\begin{equation}\label{eq.D.T}
D_1\cdot\cT_1=\deg\Bigl(R^1p_{\rho\times\tau,*}\bigl(p_{\mathscr C}^*\cE_\cT\otimes\cI_{\hat R}^{\otimes 4}\bigl)\Bigl)|_U.
\end{equation}

We proceed analyzing the terms in \eqref{eq.D.T}. As already  noticed in the proof of Proposition \ref{prop.chern}, the universal curve $\mathscr C\subset S\times \tau$ is isomorphic to $\tilde S:=Bl_{q_1,...,q_8} S$, where $q_1,...,q_8$ are the base points of $\tau$. Under this isomorphism, the universal family $\cE_{\cT_1}$ correspond to the following sheaf:
\[
\mathcal O_{\tilde S}\bigl(2(E_1+E_2+E_3+E_4)\bigl)\in \Coh(\tilde S),
\] 
where $E_1,...,E_4\subset\tilde S$ are the exceptional divisors over the points $q_1,...,q_4$ respectively.  In this setting the projection morphism
\[
q:\tilde S\cong \mathscr C\rightarrow\tau
\] 
is the morphism that at any point $s\in\tilde S$ associates the point in $\tau$ corresponding to the unique curve parametrized by $\tau$ that is individuated by the points $s,q_1,..,q_8$.

We rewrite the lower part of the diagram \eqref{diag.incidental} as follows, accordingly with the isomorphisms described: 
\[
\begin{tikzcd}
\textbf{X}:=I_{\rho\times\tau\times S} \arrow{r}{p_{\mathscr C}=:\mathbf{\hat f}} \arrow{d}[swap]{p_{\rho\times\tau}=:\mathbf{\hat q}} & \tilde S \arrow{d}{q} \\
\rho \arrow{r}{p_\tau=:\mathbf{f}} & \tau
\end{tikzcd}
\]
In the diagram above we have renamed the incidence variety $I_{\rho\times\tau\times S}$ and the map involved in the diagram, in order to simplify our notation in the computations that will come. Notice that, by definition, the divisor $\hat R\subset X$ is a section of $\hat q$. Observe that the map $q$ is flat because equidimensional surjective morphism of smooth manifolds, and that $f$ is a 4:1 flat covering, as already observed.

First, by \eqref{eq.D.T} we are interested in the sheaf $p^*_{\mathscr C}\cE_\cT\otimes\cI_{\hat R}^{\otimes 4}\in\Coh(I_{\rho\times\tau\times S})$, which corresponds under the described isomorphisms to the sheaf 
\[\hat f^*\mathcal O_{\tilde S}\bigl(2(E_1+E_2+E_3+E_4)\bigl)\otimes \cI_{\hat R}^{\otimes 4}\in\Coh(X).
\]
Observe that $X$ is smooth: $X$ is a fiber product of smooth manifolds, hence a point in $X$ is singular if the morphisms $f$ and $q$ are not smooth at that point. The morphism $q$ is not smooth on points lying on a singular curve of the pencil $\tau$, while $f$ is not smooth on ramifications points, that are the points of $\rho$ where the image of $\rho$ in $S$ is tangent to a curve in $\tau$. We conclude that $X$ is smooth if the singular conics in $\tau$ are not tangent to the rational curve $\rho_0$, which is true thanks to conditions \ref{word:condition.a} and \ref{word:condition.b} on the rational curve $\rho_0$. As consequence $\hat R$ is a Cartier divisor in $X$, and $\cI_{\hat R}^{\otimes 4}\cong \cO_X(-4\hat R)$.

Set $\hat E_i:=\hat f^*E_i$ for $i=1,...,8$, and define:
\[
L':=\mathcal O_X\bigl(2(\hat E_1+\hat E_2+\hat E_3+\hat E_4)-4\hat R\bigl)\in \Pic(X).
\]

Combining everything together, we rewrite the intersection in \eqref{eq.D.T} as follows:
\[
D_1\cdot\mathcal T_1=\deg(R^1\hat q_*L')|_U.
\]  
To compute $\deg(R^1\hat q_*L')|_U$, we are going to define a new family $L\in \Pic(X)$.
We call:
\[
\{r_i,s_i\}=C_i\cap\rho_0,\ \{x_i,y_i\}=C'_i\cap\rho_0,\ \ i=1,2,3
\]
so that $A=\{r_i,s_i,x_i,y_i\}_{i=1}^3\subseteq \rho\cong I_{\rho\times\tau}$. 

Let $\widetilde C_i,\widetilde C'_i\subset \tilde S$ be the strict transforms of $C_i,C'_i$ via the blow-up map $\tilde S\rightarrow S$, $i=1,2,3$. For any $i=1,2,3$, we consider the following curves:
\[
\{x_i\}\times \widetilde C_i,\ \{y_i\}\times \widetilde C_i,\ \{r_i\}\times \widetilde C'_i,\ \{s_i\}\times \widetilde C'_i\subseteq X.
\]
and we call:
\[
B_i:=\{x_i\}\times \widetilde C_i+\{y_i\}\times \widetilde C_i+ \{r_i\}\times \widetilde C'_i+ \{s_i\}\times \widetilde C'_i\in \Pic(X).
\]
Finally, we define:
\[
L:=\cO_X(2(\hat E_1+\hat E_2+\hat E_3+\hat E_4)+B_1+B_2+B_3-4\hat R)\in\Pic(X).
\]
Observe that $L|_U\cong L'|_U$. 
Take $p\in A$, and assume that $p\in C_1\cup C'_1$, e.g.\ $p=r_1\in C_1$ (the case of the other three intersection points is analogue); also, up to renaming the base points of \(\tau\), we can assume that $q_1,q_2\in C_1$ and $q_3,q_4\in C'_1$. Then: 
\[
L|_{\hat q^{-1}(r_1)}=L|_{\{r_1\}\times_\tau S}=L|_{\{r_1\}\times C_1\cup C'_1}\cong L''
\] 
where  $L''$ is the following line bundle on \(C_1\cup C_1'\):
\begin{equation}\label{eq.L''}
L''=\cO_{C_1\cup C_1'}(2(q_1+q_2)-4r+\alpha+2(q_3+q_4)+\beta-(q_3+q_4+\iota(q_3)+\iota(q_4)))
\end{equation}
where \(\alpha\in \mathfrak g^1_2\) on \(C_1\) and \(\beta\in \mathfrak g^1_2\) on \(C_1'\), and \(\iota\) is the involution on the \(\K3\) surface introduced in Remark \ref{rem.conics}. Notice that for any \(p\in A\) on \(C_i\cup C_i'\), \(i=1,2,3\) the restriction \(L|_{\hat q^{-1}(p)}\) is a line bundle written analogously to \(L''\), for a suitable choice of the base points of the pencil.   For a generic choice of the pencil \(\tau\) and of the \(\K3\) surface \(S\) such line bundles have no sections on \(C_i\cup C_i'\), as stated in Lemma \ref{lem.checks.D.T}. We deduce that $\deg(R^1\hat q_* L)=\deg(R^1\hat q_* L')|_U$, and then:
\[
\deg(R^1\hat q_* L)=D_1\cdot \cT_1.
\] 
We will devote the second part of the proof to the computation of $\deg(R^1\hat q_* L)$.
\\

\underline{\textit{Part 2: Computation of the intersection.}} We want to compute $c_1(R^1\hat q_*L)$. In the Grothendieck group $K(\rho)$ one has: 
\[
\hat q_!L=\hat q_*L -R^1\hat q_* L.
\] 
Notice that $\hat q_*L=0$ because it is a torsion free sheaf (as push forward of a line bundle) and with generic fiber equal to 0; it follows that:
\[
\ch(R^1\hat q_*L)=-\ch(\hat q_! L).
\]

We compute $\ch(\hat q_! L)$ using the Grothendieck-Riemann-Roch theorem on $\hat q:X\rightarrow \rho$ (we have already observed that $X$ is smooth): 
\begin{equation}\label{eq.chL}
\ch(\hat q_! L)=\hat q_*\bigl(\ch(L)\td(X)\bigl)\td(\rho)^{-1}.
\end{equation}

In what follows, we will compute the factors of equation \eqref{eq.chL}.
\begin{enumerate}
\item $\td(\rho)=\Bigl(1,\frac{c_1(\omega_\rho^\vee)}{2}\Bigl)=(1,1)$ since $\rho\cong\PP^1$; it follows that $\td(\rho)^{-1}=(1,-1)$.
\item $\td(X)=\td(T_X)=\Bigl(1,\frac{c_1}{2}, \frac{c_1^2-c_2}{12} \Bigl)$, where $c_i:=c_i(\omega^\vee_X)$ for $i=1,2$. Because of the short exact sequences
\[
0\rightarrow \hat q^*\Omega^1_{\rho}\rightarrow \Omega^1_X \rightarrow \Omega^1_{\rho / X}\cong \hat f^*\Omega^1_{\tau/ S}\rightarrow 0 
\]
and
\[
0\rightarrow q^*\Omega^1_\tau \rightarrow \Omega^1_{\tilde S}\rightarrow \Omega^1_{\tau/\tilde S}\rightarrow 0
\]
one has 
\begin{align*}
c_1(\omega_X^\vee)&=\hat q^* c_1(\omega_{\rho}^\vee)+\hat f^*[c_1(\omega^\vee_{\tilde S})+q^*c_1(\omega_\tau)] \\
&=2\xi+\hat f^*[-E_1-...-E_8-2\tilde C_t]= 2\xi+[-\hat E_1-...-\hat E_8]-8\xi \\
&=-6\xi+[-\hat E_1-...-\hat E_8].
\end{align*}
In the above computations $\xi$ is the class of a fiber of $X\rightarrow \rho$ and $\tilde C_t=q^{-1}(t)$ for $t\in\tau$,  and the intersections follow from: $\tilde S$ is the blow-up of a $\K3$ surfaces in 8 points, $\tau\cong\PP^1$ and $\hat f:X\rightarrow \tilde S$ is a covering of degree 4. Then: 
\[
c_1^2=(-6\xi)^2+12(\xi\cdot [\hat E_1]+...+\xi\cdot [\hat E_8])+[\hat E_1+...+\hat E_8]^2.
\] 
Because of the following intersections:
\begin{itemize}
\item $\xi^2=0$ since $\xi$ is a fiber;
\item $\xi\cdot \hat E_i=1$ because it is a product of a fiber and a section;
\item $\hat E_i\cdot \hat E_j=4 E_i\cdot E_j=-4\delta_{ij}$
\end{itemize}
we get 
\[
c_1^2=12\cdot 8-4\cdot 8=64.
\]

It remains to compute $c_2(T_X)$. We recall that, for a smooth projective complex manifold of dimension $d$, one has  $c_d(T_X)=\chi_{top}(X)$; this is a combination of Hirzebruch-Riemann-Roch theorem, Borel-Serre identity and the Hodge decomposition theorem. Given a ramified covering $\phi:X\rightarrow Y$ of degree $d$ with branch locus $B_\phi$ and ramification locus $R_\phi$, one has $\chi(X\setminus R_\phi)=d\cdot\chi(Y\setminus B_\phi)$, which implies 
\[
\chi(X)=d\cdot\chi(Y)+\chi(R_\phi)-d\cdot\chi(B_\phi).
\] 
In our case: $X$ is a covering of degree 4 of a blow-up of a $\K3$ surface in 8 points; the branch locus of $f:\rho\rightarrow \tau$ are the points of $\tau$ parametrizing  curves of the pencil which are tangent in $S$ to the rational curve $\rho_0$; translating the problem in $\PP^2$ via the morphism $S\rightarrow \PP^2$ associated to the linear system $|H|$ (cf. Remark \ref{rem.conics}), we need to count how many conics in a pencil are tangent to a line. This happens for 2 conics, which correspond to 2 curves on $\PP^2$ and then 2 curves of $|2H|$ pulled-back on $\hat S$; it follows that the branch locus $B_{\hat f}$ of $\hat f$ corresponds to  2 curves of genus 5. Regarding the ramification locus of $\hat f$, observe the following: each conic in $\PP^2$ tangent to the line has one point of intersection with the line, which corresponds on $S$ to 2 points of tangency among the curve of $|2H|$ and the rational curve $\rho_0$; this happens for 2 branch curves parametrized by $\tau$, which means that the ramification locus $R_f$ of $f$ consists of 4 points, and then the ramification locus $R_{\hat f}$ of $\hat f$ consists of 4 curves of genus 5. Then:
\begin{align*}
\chi(X)&=4\chi(\tilde S)+\chi(R_{\hat f})-4\chi(B_{\hat f}) \\
&=4\chi(\tilde S)-4\chi(\alpha)
\end{align*}
where $\alpha$ is a curve of genus 5. Observe that:
\begin{itemize}
\item $\chi(\tilde S)=\chi(S)+8(\chi(\PP^1)-\chi(pt))=24+8\cdot 1=32$ since $\tilde S$ is the blow up of $S$ along 8 points;
\item $\chi(\alpha)=1-10+1=-8$.
\end{itemize}
It follows  $\chi(X)=4\cdot 32+32=5\cdot 32$.
Summing up, we obtain
\[
\td(X)=\Bigl(1,\frac{-6\xi+[-\hat E_1-...-\hat E_8]}{2},-8 \Bigl).
\]
\item $\ch(L)=\Bigl(1,c_1(L),\frac{c_1(L)^2}{2}\Bigl) $,  with: 
\[c_1(L)=2(\hat E_1+\hat E_2+\hat E_3+\hat E_4)+B_1+B_2+B_3 -4\hat R,\]
\begin{align*}
c_1(L)^2&=4\Bigl(\sum_{i=1}^4\hat E_i\Bigl)^2+16\hat R^2+\Bigl(\sum_{i=1}^3B_i\Bigl)^2+4\Bigl(\sum_{i=1}^4\hat E_i\Bigl)\Bigl(\sum_{i=1}^3B_i\Bigl)\\
&-16\hat R\Bigl(\sum_{i=1}^4\hat E_i\Bigl)^2-8\hat R\Bigl(\sum_{i=1}^3B_i\Bigl).
\end{align*}
We have:
\begin{itemize}
\item $\hat R^2=-6$. Indeed, by the adjunction formula: 
\[
\hat R^2=\deg (K_{\hat R})-\deg( K_X|_{\hat R}).
\] 
Notice that $\deg(K_{\hat R})=-2$ because $\hat R\cong\PP^1$. On the other hand 
\[
K_X=\hat f^* K_{\tilde S}+R_{\hat f}
\] 
and  
\begin{align*}
\deg(\hat f^*K_{\tilde S}|_{\hat R})&=\hat f^*K_{\tilde S}\cdot \hat R=K_{\tilde S}\cdot \hat f_* \hat R \\
&=(K_S+ \sum E_i)\cdot \hat f_* \hat R.
\end{align*}
Since $S$ is a $\K3$, $K_S=0$. Furthermore, $\hat f_*\hat R=\rho_0$: by the universal property of the fiber product, $\hat f\circ \sigma$ is the inclusion of $\nu(\rho)=\rho_0$ in $\tilde S$. The curve $\rho_0$  does not intersect $E_1,...,E_8$, since $\rho_0$ does not pass through $\{q_1,...,q_8\}$; it follows $(K_S+ \sum E_i)\cdot \hat f_* \hat R=0$. 

Finally, $R_{\hat f}\cdot \hat R=4$, since the map $f$ has 4 ramification points, as noticed before. It follows $\hat R^2=-6$. 
\item $\hat R\cdot\hat E_i=\hat f_* \hat R\cdot E_i=0$ for all $i=1,...,8$.
\item $(B_1+B_2+B_3)^2=B_1^2+B_2^2+B_3^2=-24$: from $(\{r_1\}\times(\widetilde C_1\cup \widetilde C'_1))^2=0$ and  $(\{r_1\}\times\widetilde C_1)(\{r_1\}\times \widetilde C'_1)=2$ it follows \((\{r_1\}\times\widetilde C_1')^2=-2\), and similarly for the other points in $A$.
\item $(\hat E_1+\hat E_2+\hat E_3+\hat E_4)(B_1+B_2+B_3)=(E_1+E_2+E_3+E_4)\hat f_*(B_1+B_2+B_3)=(E_1+E_2+E_3)(2\sum_{i=1}^3\widetilde C_i\cup \widetilde C'_i)=24$.
\item $\hat R\cdot (B_1+B_2+B_3)=0$, since $\hat R=\{(p,p)|\ p\in\rho\}\subseteq X$ while in the definition of $B_i$ one has $r_i,s_i\notin \widetilde C'_i$ and $x_i,y_i\notin \widetilde C_i$.
\end{itemize}
We get:
\[
\ch(L)=\Bigl(1,2[\hat E_1+\hat E_2+\hat E_3+\hat E_4]+[B_1+B_2+B_3]-4[\hat R], -44 \Bigl).
\]
\end{enumerate}
Combining all the above computations with equation \eqref{eq.chL}, we finally get:
\begin{align*}
ch&(\hat q_! L)= \hat q_*\bigl(1, td_1(X)+ch_1(L), -8+td_1(X)\cdot ch_1(L)-44\bigl)\bigl(1, -1\bigl) \\
&=\hat q_*\Bigl(1,-3\xi- \frac{1}{2}\sum_{i=1}^8[\hat E_i]+2\sum_{i=1}^4[\hat E_i]+\sum_{i=1}^3[B_i] -4[\hat R], -72\Bigl)(1,-1) 
\end{align*}
where we used that $\xi\cdot \hat R=1$, since $\xi$ is a fiber and $\hat R$ is a section, and that $\xi\cdot B_i=0$, since $\xi$ is generic. Using that $\hat q_*(\xi)=\hat q_*(B_i)=0$ and $\hat q_*(\hat E_i)=\hat q_* (\hat R)=1$, we get: 
\[
\ch(\hat q_! L)=(0,-72)(1,-1)=(0,-72)
\]
i.e.\ $\ch(R^1\hat q_*L)=(0,72)$ and $D_1\cdot\mathcal T=72$.
\end{proof}

Proposition \ref{prop.D.Gamma} and Proposition \ref{prop.D.T} answer to Step $(3)$ of Strategy \ref{strategy} for what concerns the divisor $D_1$. We are ready to compute the square of $D_1$, with respect to the Beauville-Bogomolov-Fujiki form $q_{10}$.

\begin{thm}\label{thm.q(D1*)} Let $D_1\subset \M_{v_1}$ be the divisor defined in Definition \ref{def.D} and $D_1^*:=\tilde \pi_1^*D_1$, where $\tilde\pi_1:\widetilde\M_{v_1}\rightarrow \M_{v_1}$ is the symplectic desingularization. Then the class of $D_1$ in $H^2(\M_{v_1},\ZZ)$ is:
\[
D_1=-4\lambda_{v_1}(e)+52\lambda_{v_1}(f)
\]
where $e,f\in (v_1)^\perp$ are the classes defined in \eqref{eq.gen.vperp1}. Furthermore:
\[
q_{10}(D_1)=q_{10}(D_1^*)=448>0,
\] 
where $q_{10}$ is the Beauville-Bogomolov-Fujiki quadratic form on $\tilde \M_{v_1}$.
\end{thm}
\begin{proof}
This is a straightforward consequence of the computations done along the previous sections. Indeed, write: 
\[
D_1=a\lambda_{v_1}(e)+b\lambda_{v_1}(f)
\] 
with $e$ and $f$ as in \eqref{eq.gen.vperp1}. Intersecting $D_1$ with $\Gamma_1$ and \(\cT_1\) and  combining Proposition \ref{prop.chern} and Proposition \ref{prop.D.Gamma}, we get:
\begin{equation}\label{eq.D.in.vperp}
D_1=-4\lambda_{v_1}(e)+52\lambda_{v_1}(f). 
\end{equation}

It follows :
\[
q_{10}(D_1^*)=q_{10}(D_1)=(-4e+54 f)^2=448,
\]
since $e^2=2$, $<e,f>=-1$ and $f^2=0$ in the Mukai lattice $H^*(S,\ZZ)$.
\end{proof}

As observed in Remark \ref{rem.pullback.strict.trans}, when the multiplicity \(m\) of \(\Sigma_1\) in \(D_1\) is greater than 1, the divisor \(D_1^*\) is ruled by non reduced curves. We do not want this case to occur, because we would like to apply Corollary \ref{cor.defo.curves} afterwords. We start checking if $\Sigma_1$ is contained in $D_1$. 

\begin{prop}\label{lem.sing.locus.D1} The divisor $D_1$ does not contain the singular locus $\Sigma_1=\M_{v_1}\setminus \M_{v_1}^s$. It follows that $D_1^*=\widetilde D_1$, where \(\widetilde D_1\) is the strict transform of the divisor \(D_1\) via the symplectic desingularization \(\widetilde \pi_1:\widetilde\M_{v_1}\to\M_{v_1}\).
\end{prop}
\begin{proof}
A general point $p\in \Sigma_1$ is an $S$-equivalence class of sheaves supported on a reducible curve $C=C_1\cup C_2$, with $C_i\in|H|$ smooth, $i=1,2$, and $C_1\cap C_2=\{n_1,n_2\}$, with $n_1\neq n_2$; it has polystable representative $F_1\oplus F_2$ with $F_1$ and $F_2$ non-isomorphic stable sheaves supported on $C_1$ and $C_2$ respectively, and with $\mathfrak v(F_1)=\mathfrak v(F_2)=\frac{v_1}{2}=(0,h,2)$ (see \cite{OG10} and the proof of Proposition 5.2 in \cite{lehnsorger}). As consequence, for $i=1,2$ the sheaf $F_i$ is push-forward of a (general) torsion-free sheaf of rank 1 on $C_i$, hence a line bundle $L_i$ because of the smoothness of $C_i$; also, $\deg(L_i)=3$ since $\mathfrak v(F_i)=(0,h,2)$. We want to prove that \(p\notin D_1\); \(D_1\) is defined as the closure in \(\M_{v_1}\) of a divisor on \(\cJ^{8}_{|2H|^{sm}}\), and \(p\in \M_{v_1}\smallsetminus \cJ^{8}_{|2H|^{sm}}\).

Because of the moduli-theoretic interpretation given by O'Grady in \cite[Section 2.2]{OG6}, for \(p\in \Sigma_1\) general there exists a 1:1 correspondence between the sheaves in the \(S\)-equivalence class \([F_1\oplus F_2]\) and simple semistable sheaves on \(C_1\cup C_2\) up to \(\widetilde S\)-equivalence; in particular,  any sheaf in \(\Ext^1(F_2,F_1)\) is identified with a sheaf  in \(\Ext^1(F_1,F_2)\).

Fix a point \(r\in C_1\cap\rho_0\). Any sheaf in the \(S\)-equivalence class  \([F_1\oplus F_2]\) corresponds to a simple sheaf \(F\) fitting in the following short exact sequence:
\[
0\rightarrow F_1\rightarrow F\rightarrow F_2\rightarrow 0.
\]
We call \(G_i\) the torsion free part of \(F|_{C_i}\), \(i=1,2\);  from the sequence above it follows that
\[
G_2\cong L_2,\ \ G_1\cong \widetilde L_1
\]
where \(\widetilde L_1\) is isomorphic to \(L_1(n_i)\) or \(L_1(n_1+n_2)\), depending on the rank of \(F\) on the nodes; notice that \(F\) can not have rank two on both the nodes, because by assumption \(F\) is a simple sheaf. It follows that $\widetilde L_1$ is a line bundle of degree 4 or 5 on $C_1$, then $\deg(\widetilde L_1(-4r))=0,1$, for $i=1,2$. Since $L_1$ is general and $g(C_1)=2$, one has in any case that $h^0(C_1,\widetilde L_1(-4r))=0$. By \cite[Lemma 1.0.7]{rapagnettaOG10} we have the following short exact sequence:
\[
0\to F(-4r)\to i_{1*}\widetilde L_1(-4r)\oplus i_{2*}L_2\to Q\to 0
\]
where \(i_j:C_j\hookrightarrow S\) is the inclusion of the curve and \(Q\) is the structure sheaf on \(C_1\cap C_2\) if \(F\) has rank 1 on both the nodes, and it is the structure sheaf on \(n_j\) if \(F\) has rank 2 on one node \(n_i\), where \(j\neq i\); observe that \(L_2(-4r)\cong L_2\) because \(r\notin C_2\). It follows:
\begin{equation}\label{eq.global.sections}
0 \to H^0(S,F(-4r))\to  H^0(S,i_{1*}\widetilde L_1(-4r))H^0(S,i_{2*}L_2)\to H^0(S,Q).
\end{equation}
When \(F\) has rank 1 on both the nodes \(n_1\) and \(n_2\), one has \(H^0(S,Q)=H^0(S,\CC_{n_1}\oplus \CC_{n_2})\cong \CC^2\), and any section \(\sigma\in H^0(C,F(-4r))\) is obtained gluing the zero section on \(C_1\) and a section in \(H^0(C_2,L_2(-n_1-n_2))\); from \(\deg(L_2)=3\) and \(L_2\) generic it follows \(\sigma=0\). By Remark \ref{rem.D.schematic} and Remark \ref{rem.D1eD1'} we conclude that in this case \(F\) is not obtained as a limit of sheaves in \(D_1\cap \cJ^8_{|2H|^{sm}}\).

When \(F\) has rank 2 on one of the two nodes, e.g.\ on \(n_1\), one has \(H^0(S,Q)=H^0(S,\CC_{n_2})\cong \CC\); in this case, from \eqref{eq.global.sections} any section \(\sigma\in H^0(S,F(-4r))\) is obtained gluing the zero section on \(C_1\) and a section in \(H^0(S,i_{2*}L_2(-n_2))\). From \(\deg (L_2(-n_2))=2\) and \(L_2\) general it follows \(h^0(C_2,L_2(-n_2))=1\), hence in this case \(H^0(S,F(-4r))=\CC\cdot \sigma\), where \(\sigma\) is obtained gluing the zero section on \(C_1\) and the (unique, up to scalar) non-zero section on \(C_2\) vanishing on \(n_2\). We want to conclude that also in this case \(F\) is not obtained as a limit of sheaves in \(D_1\cap \cJ^8_{|2H|^{sm}}\). Assume that there exists a 1-dimensional family of sheaves \(\cF_\Delta\to \Delta\), with \(\cF_0=F\) for some \(0\in \Delta\) and \(\cF_d\in D\cap \cJ^8_{|2H|^{sm}}\) for any \(d\in \Delta\smallsetminus \{0\}\). Consider the following incidence variety:
\[
I:=\{(r,[F])\in \rho_0\times \M_{v_1}\ |\ r\in \Supp(F)\}\subseteq \rho_0\times \M_{v_1}
\]
and call \(I':=I\cap \rho_0\times\cJ^8_{|2H|^{sm}}\subseteq\rho_0\times\M_{v_1}\). If we call:
\[
A:=\{(r,[F])\in I'\ | \ h^0(S,F(-4r))>0\}\subseteq\rho_0\times\M_{v_1}
\]
we have that \(D_1=\overline{p(A)}^{\M_{v_1}}=p\bigl(\overline A^{\rho_0\times\M_{v_1}}\bigl)\), where \(p:\rho_0\times\M_{v_1}\to \M_{v_1}\) is the projection; here we used that \(p\) is closed, since \(\rho_0\) is complete.  It follows that \(\cF_\Delta\) lifts to a 1-dimensional family \((r_\Delta,\cF_\Delta)\to \Delta\) with \((r_d,\cF_d)\in \overline A^{\M_{v_1}}\) for any \(d\in \Delta\) and \((r_0, \cF_0)=(r,F) \) for some \(0\in \Delta\); observe that this implies \(h^0(\cF_d(-4r_d))\ge 1\) for any \(d\in \Delta\). We have seen that \(h^0(F(-4r))=1\); up to working in a small analytic neighbourhood of \(F\in \M_{v_1}\), we can assume that \(h^0(\cF_d(-4r_d))= 1\) for any \(d\in \Delta\). We call \(\xi_d\) the vanishing locus of this unique section; we obtain for any \(d\in \Delta\smallsetminus \{0\}\) a well-defined cycle \(4r_d+\xi_d\in J^8(\Supp(\cF_d))\), with \(\xi_d\) effective and of degree 4 on the smooth curve \(\Supp(\cF_d)\). For \(d\to 0\) the family \(\{4r_d+\xi_d\}_{d\in \Delta\smallsetminus \{0\}}\) has limit an effective cycle on \(C=C_1\cup C_2\), which is by definition the sum of the cycle \(4r\) and the zero locus of the unique (up to scalars) section \(\sigma\) of \(H^0(C,F(-4r))\); this contradicts the vanishing of \(\sigma\) on the curve \(C_1\). We conclude that \(F\) can not be written as a limit of sheaves in \(D_1\cap \cJ^8_{|2H|^{sm}}\).

Fix now a point \(r\in C_2\cap\rho_0\).  In this case we can take \(F\) as an extension in \(\Ext^1(F_1,F_2)\), and proceed as before to conclude that \(F\) is not obtained as a limit of sheaves in \(D_1\cap \cJ^8_{|2H|^{sm}}\). It follows that \(p=[F]\notin D_1\), and the statement of the Proposition is proved.
\end{proof}

\begin{rem} In Remark \ref{rem.D1eD1'} we have observed that given a sheaf \([F]\in\M_{v_1}\) such that \(h^0(S,F(-4r))>0\) for some \(r\in \Supp(F)\cap \rho_0\) then \([F]\in D_1'\), but we can not conclude \([F]\in D_1\). An example of that is in the proof of Proposition \ref{lem.sing.locus.D1}, when the sheaf \(F\) has rank 2 on one on the two nodes in \(C_1\cap C_2\). 
\end{rem}

\begin{cor}\label{cor.D1*.ruled.pos} The divisor $D_1^*=\widetilde D_1$ introduced in Theorem \ref{thm.q(D1*)} and the divisor \(B_1:=\widetilde D_1+\widetilde \Sigma_1\) are positive uniruled divisors on the $\OG10$ manifold $\widetilde\M_{v_1}$, ruled by reduced curves; here \(\widetilde \Sigma_1\) is the exceptional divisor of the symplectic desingularization \(\widetilde\pi_1:\widetilde\M_{v_1}\to\M_{v_1}\). If \(\Gamma_{v_1}\) is the lattice defined in \eqref{eq.Gamma_v} and \(f_{v_1}:\Gamma_{v_1}\xrightarrow{\sim}H^2(\widetilde\M_{v_2},\ZZ)\) is the isometry introduced in Theorem \ref{thm.Gamma_v}, then:
\[
f_{v_1}^{-1}(\widetilde D_1)=\bigl((-4,-4h,52),0\bigl)\in\Gamma_{v_1}
\]
\[
f_{v_1}^{-1}(B_1)=\bigl((-4,-4h,52),\sigma\bigl)\in\Gamma_{v_1}.
\]
Furthermore:
\[
q_{10}(\widetilde D_1)=448>0, \ q_{10}(B_1)=442>0.
\]
\end{cor}
\begin{proof}
This is a straightforward consequence of Theorem \ref{thm.q(D1*)} and Proposition \ref{lem.sing.locus.D1}, and of the explicit description of the isometry \(f_{v_1}\) given in Theorem \ref{thm.Gamma_v}.
\end{proof}

\subsection{The class of $\widetilde D_2$ and of \(B_2\)}\label{subsection:class.D2}

We pass to the divisor $D_2\subseteq\M_{v_2}$, introduced in Definition \ref{def.D2}. The main result of the section is the following theorem.

\begin{thm}\label{thm.q(tildeD2)} Let $D_2$ be the divisor introduced in Definition \ref{def.D2}, $\widetilde D_2\subseteq\widetilde\M_{v_2}$ the strict transform of $D_2$ via the symplectic desingularization $\tilde\pi_2:\widetilde\M_{v_2}\rightarrow \M_{v_2}$, and \(B_2:=\widetilde D_2+\widetilde \Sigma_2\), where \(\widetilde \Sigma_2\) is the exceptional divisor of \(\widetilde \pi_2\). Furthermore, let \(\Gamma_{v_2}\) be the lattice defined in \eqref{eq.Gamma_v}, and \(f_{v_2}:\Gamma_{v_2}\xrightarrow{\sim} H^2(\widetilde\M_{v_2},\ZZ)\) the isometry introduced in Theorem \ref{thm.Gamma_v}. Then:
\[
f_{v_2}^{-1}(\widetilde D_2)=\bigl((-4,-2h,22),-2\sigma\bigl)\in \Gamma_{v_2}
\] 
\[
f_{v_2}^{-1}(B_2)=\bigl((-4,-2h,22),-\sigma\bigl)\in \Gamma_{v_2}
\] 
and 
\[
q_{10}(\widetilde D_2)=160>0,\ q_{10}(B_2)=178>0.
\]
\end{thm}

Theorem \ref{thm.q(tildeD2)} will be proved following  Strategy \ref{strategy}, as done in the previous subsection for the divisor $D_1^*$. Nevertheless, the case of the divisor $D_2$  is a bit more delicate, because the space $\M_{v_2}$ is not locally factorial, but only 2-factorial; this is an immediate consequence of Theorem \ref{thm.factorial}. It follows that, before applying Strategy \ref{strategy}, we need to check that the divisor $D_2$ is a Cartier divisor. 

\begin{lem}\label{lemma.D2containsSigma} The divisor $D_2$ contains the singular locus $\Sigma_2:=\M_{v_2}\setminus\M_{v_2}^s$.
\end{lem}

\begin{proof}
We start as in the proof of Lemma \ref{lem.sing.locus.D1}: a general point $p\in \Sigma_2$ is a $S$-equivalence class of sheaves supported on a curve $C=C_1\cup C_2$, with $C_i\in|H|$ smooth, $i=1,2$, and $C_1\cap C_2=\{n_1,n_2\}$, with $n_1\neq n_2$; it has polystable representative $F_1\oplus F_2$ with $F_1$ and $F_2$ non-isomorphic stable sheaves supported on $C_1$ and $C_2$ respectively, and $\mathfrak v(F_1)=\mathfrak v(F_2)=\frac{v_2}{2}=(0,h,1)$ (see \cite{OG10} and the proof of Proposition 5.2 in \cite{lehnsorger}). As consequence, for $i=1,2$ the sheaf $F_i$ is the push-forward of a (general) torsion-free sheaf of rank 1 on $C_1$, hence a line bundle $L_i$ because of the smoothness of $C_i$; also, $\deg(L_i)=2$ since $\mathfrak v(F_i)=(0,h,1)$. We want to prove that such a general $p\in\Sigma_2$ belongs to $D_2$.

By generality of $L_2$, we can assume that it is effective, since $g(C_2)=2$. Fixed $r\in C_1\cap \rho$, any line bundle of degree 2 on $C_1$ can be written as $x+2r-n_1-n_2$, with $x\in J^2(C_1)$; this is true because $x\mapsto x+2r-n_1-n_2$ is an isomorphism of $J^2(C_1)$. Furthermore, we can assume that $x$ is effective because $g(C_1)=2$. Summing up, we got that we can write a general $p\in \Sigma_2$ as 
\begin{equation}\label{eq.pinSigma}
p=[j_{1,*}\cO_{C_1}(2r+y_1+y_2-n_1-n_2)\oplus j_{2,*}\cO_{C_2}(y_3+y_4)],
\end{equation}
for some $y_1,y_2\in C_1$ and $y_3,y_4\in C_2$, where $j_i:C_i\hookrightarrow S$ is the inclusion, $i=1,2$;  we want to prove that such a point belongs to $D_2$. By generality of $p\in\Sigma$, the points $y_1,...,y_4$ that we got in the expression \eqref{eq.pinSigma} determine a pencil $Q\subset|2H|$, and a point in $\rho$ determines one curve of the pencil. Let us consider the incidence variety $\mathscr C\subset S\times\rho$, $\mathscr C:=\{(y,x)\in S\times \rho\ |\ y\in C_x\}$, with $C_x$ the unique curve of $Q$ determined by $x\in \rho$; let $i:\mathscr C\hookrightarrow S\times \rho$ be the inclusion and $\Delta\subset S\times\rho$ the diagonal. The curve parametrized by the family $i_*\cO_{\mathscr C}(2\Delta+y_1\times\rho+...+y_4\times\rho)$ is contained in the divisor $D_2$ (on the smooth curves of $Q$ it consists of line bundles as in the definition of $D_2$), and choosing $r\in\rho$ as in \eqref{eq.pinSigma} we get back the class of the point $p$. We conclude that $p\in D_2$ and then $\Sigma_2\subset D_2$.
\end{proof}

Since the singular locus $\Sigma_2$ is contained in the divisor $D_2$, the divisor $D_2$ can actually happen to be non Cartier. However this is not the case, as stated in the following result.

\begin{prop}\label{prop.D.Cartier} $D_2$ is a Cartier divisor.
\end{prop}
\begin{proof}

Let $\widetilde D_2$ be the strict transform of $D_2$ via the symplectic resolution $\tilde\pi_2:\widetilde\M_{v_2}\rightarrow \M_{v_2}$, and let $\delta\subset\widetilde\M_{v_2}$ be the fiber of $\tilde\pi_2$ over a general point $p\in\Sigma_2$. A divisor $D_2\subset\M_{v_2}$ is Cartier if and only if the intersection $\widetilde D_{v_2}\cdot\delta$ is even: let $n$ be the smallest positive integer such that $nD_2$ is Cartier, and let $m$ be the multiplicity of $\Sigma_2$ in $nD_2$, i.e.\ the multiplicity of a general $p\in\Sigma_2$ in $D_2$. Then: 
\[
\tilde\pi_2^*(nD_2)=n\tilde D_2+m\tilde\Sigma_2.
\] 
We intersect the divisors in the expression above with $\delta$: $\tilde\pi_2^*(nD_2)\cdot\delta=0$ by the projection formula because $\delta$ gets contracted by $\tilde\pi_2$, and $\tilde\Sigma_2\cdot\delta=-2$ because $\widetilde\M_{v_2}$ has trivial canonical bundle (cf. proof of Theorem 2.0.8 in \cite{rapagnettaOG10}). We get:
\[
n\tilde D_2\cdot\delta=2m.
\] 
It follows that if the intersection $\tilde D_2\cdot\delta$ is an odd number, the number $n$ needs to be an even number, ore more precisely $n=2$, since as already noticed $n$ can be just 1 or 2, see Theorem \ref{thm.factorial}; in other words, if $\widetilde D_2\cdot\delta$ is an odd number then $D_{v_2}$ is not a Cartier divisor. On the other hand, if $\widetilde D_2\cdot\delta=2a$ then $m=an$ and 
\[
\tilde\pi_2^*(nD_2)=n\tilde D_2+an\tilde\Sigma_2=n(\widetilde D_2+a\tilde\Sigma_2)
\] 
where $(\widetilde D_2+a\tilde\Sigma_2)\cdot\delta=0$. This last equality implies that $\widetilde D_2+a\tilde\Sigma_2$ is the pullback of some Cartier divisor in $\M_{v_2}$, then $D_2$ is Cartier.

We are going to compute the intersection $\widetilde D_2\cdot \delta$. Fix $p\in\Sigma_2$ general; as already noticed at the beginning of the proof of Lemma \ref{lemma.D2containsSigma}, we can assume that $p$ is the $S$-equivalence class of sheaves with polystable representative $F_1\oplus F_2$, with $\mathfrak v(F_i)=(0,h,1)$ and $\Supp(F_i)=C_i\in |H|$ smooth curve; we call $C:=C_1\cup C_2$ and $C_1\cap C_2=:\{n_1,n_2\}$. The curve \(\delta\) is isomorphich to \(\PP^1\) (\cite[Claim 2.2.4 and 3.0.10]{OG10}), and it parametrizes extensions in \(\Ext^1(F_2,F_1)\); it is  defined by the following extension sheaf:
\[
0\rightarrow p^*F_1\otimes q^*\cO_{\delta}(1)\rightarrow \cE_\delta\rightarrow p^*F_2\rightarrow 0
\]
where $p:S\times\delta\rightarrow S$ and $q:S\times\delta\rightarrow \delta$ are the projections (cf. \cite[Example 2.1.12]{huy.lehn}). 

As done for the intersection $D_1\cdot\cT_1$ in Proposition \ref{prop.D.T}, in order to compute the intersection $\widetilde D_2\cdot \delta$ we divide the computation in two parts: first, we give an easier to compute reformulation of the intersection; secondly, we proceed with the actual computation. \\

\underline{\textit{Reformulation of the intersection} $D_2\cdot\delta$.} As first remark, we recall that the divisors $D_2$ and $D'_2$ have been defined as closure in $\M_{v_2}$ of divisors defined in $\cJ^6_{|2H|^{sm}}$ and $\M_{v_2}^s$ respectively, see Remark \ref{rem.D2.schematic}. As consequence, the strict transforms $\widetilde D_2$ and $\widetilde D'_2$ are defined as the closure in $\widetilde \M_{v_2}$ of the same divisors, since the desingularization $\widetilde \pi_2:\widetilde\M_{v_2}\rightarrow\M_{v_2}$ is an isomorphism on $\widetilde\pi^{-1}(\M_{v_2}^s)\supseteq \widetilde\pi^{-1}(\cJ^6_{|2H|^{sm}})$. 

As usual, we denote by $\nu:\rho\rightarrow\rho_0$ the normalization of the rational curve $\rho_0\in |2H|$ used to define $D_2$ in Definition \ref{def.D2}. Consider the following commutative diagram:
\[
\begin{tikzcd}
I_{\rho\times \widetilde \M_{v_2}} \arrow{r}{p_{\widetilde\M_{v_2}}}  & \widetilde\M_{v_2} \\
I_{\rho\times\delta} \arrow{r}{p_\delta}\arrow[hook, u, "\hat j"] & \delta \arrow[hook, u, "j"]
\end{tikzcd}
\]
where: $j:\delta\hookrightarrow\widetilde\M_{v_2}$ is the inclusion given by the family $\cE_\delta$, the variety \(I_{\rho\times\delta}\) is the incidence variety
\[I_{\rho\times\delta}:=\{(r,F)\in \rho\times\delta|\ \nu(r)\in\Supp(F)=C\}
\]
with projection $p_\delta:I_{\rho\times\delta}\rightarrow\delta$, and \(I_{\rho\times\widetilde\M_{v_2}}\) is the incidence variety
\[
I_{\rho\times\widetilde\M_{v_2}}:=\{(r,[F])\in\rho\times\widetilde\M_{v_2}\ |\ \nu(r)\in\Supp(F)\}
\]
with projection \(p_{\widetilde\M_{v_2}}:I_{\rho\times\widetilde\M_{v_2}}\to\widetilde\M_{v_2}\). Notice that the diagram is commutative. If $\rho_0\cap C_1=\{r_1,r_2\}$ and $\rho_0\cap C_2=\{r_3,r_4\}$, then $I_{\rho\times\delta}\cong\{r_1,r_2,r_3,r_4\}\times\delta$. Consider the inclusion of incidence varieties $I_{\rho\times\M_{v_2}^s}\subseteq I_{\rho\times\widetilde\M_{v_2}}$, and the family $\hat\cF\otimes\cI_R^{\otimes 2}\in\Coh(I_{\rho\times S\times\M_{v_2}^s})$ introduced in Remark \ref{rem.D2.schematic}; with the same argument used in the proof of Proposition \ref{prop.D.T}, we get that:
\begin{equation}\label{eq.first.inequal.D2}
\widetilde D_2\cdot\delta \le \widetilde D'_2\cdot\delta=\deg(j^*\widetilde D_2')=\deg(\hat j^*\overline{\Supp R^1p_{\rho\times\M^s_{v_2}*}\hat \cF\otimes\cI_R^{\otimes 2}}^{I_{\rho\times\widetilde\M_{v_2}}}),
\end{equation}
where $p_{\rho\times\M^s_{v_2}}: I_{\rho\times S\times\M^s_{v_2}}\rightarrow I_{\rho\times\M^s_{v_2}}$ is the projection. Again similarly as done in the proof of Proposition \ref{prop.D.T} in \eqref{diag.incidental}, we consider the following diagram:
\begin{equation}\label{eq.D2.delta}
\begin{tikzcd}
I_{\rho\times S\times\widetilde\M_{v_2}} \arrow{r}{p_{\rho\times\widetilde\M_{v_2}}}  & I_{\rho\times\widetilde\M_{v_2}} \\
I_{\rho\times S \times \delta}\cong \{r_1,r_2,r_3,r_4\}\times C\times \delta \arrow{r}{p_{\rho\times\delta}}\arrow[hook, u, "\tilde j"]\arrow{d}{p_{C\times \delta}} & I_{\rho\times\delta}\cong \{r_1,r_2,r_3,r_4\}\times \delta \arrow[hook, u, "\hat j"]\arrow{d}{p_\delta} \\
I_{S\times\delta}\cong C\times\delta \arrow{r}{p_\delta} & \delta
\end{tikzcd}
\end{equation}
where:
\begin{itemize}
\item $I_{\rho\times S\times\delta}:=\{(r,x,F)\in\rho\times S\times\delta|\ \nu(r),x\in\Supp(F)=C, F\in\delta \}\cong \{r_1,r_2,r_3,r_4\}\times C\times \delta$ and $p_{\rho\times\delta}:I_{\rho\times S\times\delta}\rightarrow I_{\rho\times\delta}$ is the projection;
\item $I_{S\times\delta}:=\{(x,F)\in S\times\delta|\ x\in\Supp(F)=C\}\cong C\times\delta$, and $p_\delta:C\times\delta\rightarrow\delta$ is the projection. 
\end{itemize}
Then:
\begin{align}
\widetilde D_2\cdot\delta&\le\deg(\hat j^*\overline{\Supp R^1p_{\rho\times\M^s_{v_2}*}\hat \cF\otimes\cI_R^{\otimes 2}}^{I_{\rho\times\widetilde\M_{v_2}}})\nonumber \\ 
&=\deg\overline{(R^1p_{\rho\times\delta*}\tilde j^*\hat\cF(-2R))}^{I_{\rho\times\delta}} \nonumber \\ 
&\le\deg\Bigl(R^1p_{\rho\times\delta,*}(p^*_{C\times\delta}\cE_\delta(-2\tilde j^* R))\Bigl) \label{eq.secon.inequal.D2}
\end{align}
where the first inequality is \eqref{eq.D2.delta}, the equality follows from the base change and the fact that we are working with the top cohomology, and the last inequality follows from the upper semicontinuity of the function $h^0$; here the sheaf $\cE_\delta$ is considered as a sheaf on $C\times\delta$, where it has support. Recall that $I_{\rho\times\delta}\cong\{r_1,r_2,r_3,r_4\}\times\delta$;  for any $i=1,...,4$ we consider the following diagram, obtained as restriction of the lower part of the diagram in \eqref{eq.D2.delta} (we renamed the arrows we are going to use for simplicity):
\[
\begin{tikzcd}
\{r_i\}\times C\times\delta\cong C\times\delta \arrow{r}{f}  \arrow{d}{g} & \{r_i\}\times\delta\cong\delta \arrow[d] \\
C\times\delta \arrow{r}{} & \delta 
\end{tikzcd}
\]
we get:
\begin{align}
\widetilde D_2\cdot\delta&\le\deg\Bigl(R^1p_{\rho\times\delta,*}(p^*_{C\times\delta}\cE_\delta(-2\tilde j^* R))\Bigl)=\sum_{i=1}^4 \deg R^1f_*(g^*\cE_\delta\otimes \cI^2_{r_i}) \nonumber \\
&=\sum_{i=1}^4\deg R^1 q_*(\cE_\delta\otimes\cI^2_{r_i}=:\cH_{r_i}). \label{eq.sum.D2.delta}
\end{align}
We want to read the inequalities written so far, to understand how to get $\widetilde D_2\cdot \delta$. The inequality in \eqref{eq.first.inequal.D2} comes from the inclusion $\widetilde D_2\subseteq\widetilde D'_2$: there are sheaves $[F]\in \widetilde D'_2$ possibly not in $\widetilde D_2$. The inequality in \eqref{eq.secon.inequal.D2} comes from the fact that on the right we are counting all sheaves $[F]\in\delta$ such that $h^0(F(-2r))>0$ for some $r\in C\cap\rho_0$: we know that any sheaf $[F]\in \widetilde D_2$ has this property, but this is not a characterization of sheaves in $\widetilde D_2$ (see Remark \ref{rem.D1eD1'}). We deduce that both inequalities are due exactly to those sheaves $[F]\in\delta$ such that $h^0(F(-2r))>0$ for some $r\in C\cap\rho_0$ but $[F]\notin \widetilde D_2$, which contribute to the sum in \eqref{eq.sum.D2.delta} but not to $\widetilde D_2\cdot\delta$. We are going to understand which sheaves in $\delta$ have this property, in order to exclude them from the computation and get the intersection $\widetilde D_2\cdot\delta$.

Take $r\in C_1\cap\rho_0$. Because of the moduli-theoretic interpretation given by O'Grady in \cite[Section 2.2]{OG6}, there is a 1:1 correspondence between the sheaves in \(\delta\) and simple semistable sheaves on \(C_1\cup C_2\) up to \(\widetilde S\)-equivalence; in particular,  any sheaf in \(\Ext^1(F_2,F_1)\) is identified with a sheaf  in \(\Ext^1(F_1,F_2)\). Then we take $F\in\Ext^1(F_2,F_1)$ simple (hence $[F]\in\delta$) such that $h^0(S,F(-2r))>0$; we call $\sigma$ a global non-zero section of $F(-2r)$. Since \(p\in\Sigma_2\) is general, it holds \(F_i=j_{i*}L_i\) with \(L_i\) general line bundle on \(C_i\) of degree 2, \(i=1,2\); here, \(j_i:C_i\hookrightarrow S\) is the inclusion. Finally, for $i=1,2$ we name $G_i$ the torsion free part of the restriction  $F|_{C_i}$. We have:
\[
G_2\cong L_2,\ G_1\cong \widetilde L_1,
\] 
where \(\widetilde L_1\) is the line bundle  \(L_1(n_i)\) or \(L_1(n_1+n_2)\) depending on the rank of \(F\) on the nodes \(n_1\) and \(n_2\); notice that \(F\) can not have rank two on both \(n_1\) and \(n_2\), since by assumption it is a simple sheaf. It follows that \(\deg(G_1(-2r))=1\) or \(2\), depending on \(\widetilde L_1\). When \(\deg(G_1(-2r))\) equals 1, i.e.\ when \(F\) has rank two on one node, then the restriction \(\sigma|_{C_1}\) of the section \(\sigma\) above vanishes, because it is a section of a general line bundle of degree 1 on a curve of genus 2; it follows that \(\sigma\)  vanishes on the node $n_i$ and then on $C_2$ as well. We conclude that it is the zero section on $C$, hence this case can not occur if $h^0(S,F(-2r))>0$. Assume now that $F$ is locally free, hence \(G_1\cong L_1(n_1+n_2)\) and \(\deg(G_1(-2r))=2\). The hypothesis \(h^0(S,F(-2r))>0\) implies that \(\sigma\) is not null on $C_1$ and $C_2$: if we assume that \(\sigma|_{C_i}=0\), then arguing as before we get that for degree reasons \(\sigma\) needs to be null on the other component as well.  Notice that $G_2(-2r)\cong L_2$ because \(r\notin C_2\), then by $p\in\Sigma_2$ general we have $h^0(C_2,G_2(-2r))=1$, hence $H^0(C_2,G_2(-2r))\cong \CC\cdot\sigma|_{C_2}$; we call $G_2(-2r)\cong \cO_{C_2}(x_1+x_2)$, with $x_1,x_2\in C_2$ distinct points non associated via the involution $\iota$ on $S$ introduced in Remark \ref{rem.conics} (otherwise we would have $\cO_{C_2}(x_1+x_2)\cong\omega_{C_2}$, and $h^0(C_2,\omega_{C_2})=2$). On the other hand, $G_1(-2r)\cong L_1(n_1+n_2-2r)$, then again by $p\in\Sigma_2$ general we can assume $G_1\cong \cO_{C_1}(y_1+y_2)$, with $y_1,y_2\in C_1$ distinct points non associated via the involution $\iota$ on $S$. Under these assumptions, the points $x_1,x_2,y_1,y_2$ are in general position on $S$ and we can proceed as in the proof of Lemma \ref{lemma.D2containsSigma} to construct a family of sheaves in $\widetilde D_2$ having $F$ as limit. We conclude that $F$ is in $\widetilde D_2$.  

In conclusion, the extensions $F\in\Ext^1(F_2,F_1)$ contributing to the sum in \eqref{eq.sum.D2.delta} because of $r\in C_1\cap \rho_0$ do contribute to the intersection $\widetilde D_2\cdot\delta$ as well. Regarding the points \(r\in C_2\cap \rho_0\), because of the moduli-theoretic interpretation mentioned above we can take \(F\in\delta\) as an extension in \(\Ext^1(F_1,F_2)\), and conclude as before that the extensions $F\in\Ext^1(F_1,F_2)$ contributing to the sum in \eqref{eq.sum.D2.delta} because of $r\in C_2\cup\rho_0$, do contribute to the intersection $\widetilde D_2\cdot\delta$ as well.  

In conclusion, if we call $\cH'_\delta=\cE'_\delta\otimes\cI^2_{r_i}$ and $\cE'_\delta$ is a sheaf defined as an extension similarly to $\cE_\delta$, but with $F_1$ and $F_2$ swapped:
\[
0\rightarrow p^*F_2\otimes q^*\cO_\delta(1)\rightarrow \cE_\delta\rightarrow p^*F_1\rightarrow 0
\] 
we conclude that:
\begin{align*}
\widetilde D_2\cdot\delta&=\sum_{i=1}^2\deg (R^1q_*\cH_{r_i})+\sum_{i=3}^4\deg (R^1q_*\cH'_{r_i}) \\
&=-\sum_{i=1}^2\deg(q_!\cH_{r_i})-\sum_{i=3}^4\deg(q_!\cH'_{r_i})
\end{align*}
where the last equality follows from the expression $q_!\cH_{r_i}=q_*\cH_{r_i}-R^1q_*\cH_{r_i}$ and the fact that $q_*\cH_{r_i}$ is a torsion free sheaf with general fiber equal to 0. 

In what remains, we are going to compute the summand on the right term of the equality above.\\

\underline{\textit{Computation of the intersection.}} The computation of $\deg(q_!\cH_{r_i})$ and $\deg(q_!\cH'_{r_i})$ is exactly the same for the four points $r_1,r_2,r_3,r_4$; for this reason, we write the computation only for $r:=r_1\in C_1$, and then we will multiply the result by 4 to obtain the intersection $-\widetilde D_2\cdot \delta$. Observe that in the case of $r=r_1$ we want to consider the family $\cE_\delta$ defined by the following extension: 
\[
0\rightarrow p^*F_1\otimes q^*\cO_\delta(1)\rightarrow \cE_\delta\rightarrow p^*F_2\rightarrow 0.
\]
We call $\cH:=\cH_{r}$. By the Grothendieck-Riemann-Roch theorem on the projection $q:S\times\delta\rightarrow\delta$, we need to compute:
\[
\ch(q!\cH)=q_*\bigl(\ch(\cH)\td(S\times\delta)\bigl)\td^{-1}(\delta),
\] 
where:
\begin{itemize}
\item $\td(\delta)^{-1}=(1,1)^{-1}=(1,-1)$ because $\delta\cong\PP^1$.
\item $\td(S\times\delta)=p^*\td(S)q^*\td(\delta)$, where $\td(S)=(1,0,2)$ because $S$ is a K3 surface; then: 
\[
\td(S\times\delta)=(1,[S\times pt],2[pt\times\delta],2).
\]
\end{itemize}
We pass to $\ch(\cH)$. The sheaf $\cH$ fits in the following short exact sequence:
\[
0\rightarrow \cH\rightarrow \cE_\delta\rightarrow \cE_\delta|_{2r\times\delta}\rightarrow 0.
\]
Furthermore, because of the assumption $r\in C_1$, from the short exact sequence defining $\cE_\delta$ we deduce:
\[
0\rightarrow p^*F_1\otimes q^*\cO_\delta(1)|_{r\times\delta}\cong\cO_{r\times\delta}(1)\xrightarrow{\sim} \cE_\delta|_{r\times\delta}\rightarrow 0.
\]
Then:
\begin{align*}
\ch(\cH)&=\ch(\cE_\delta)-2\ch(\cO_{r\times\delta}(1))=p^*\ch(F_1)q^*\ch(\cO_\delta(1))+p^*\ch(F_2)-2\ch(\cO_{r\times\delta}(1)) \\
&=(0,2[h\times\delta],[h\times pt],-1).
\end{align*}
In conclusion, we have:
\begin{align*}
\ch(q_!\cH)&=q_*\Bigl((0,2[h\times\delta],[h\times pt],-1)(1,[S\times pt],2[pt\times\delta],2)\Bigl)(1,-1) \\
&=(0,-1)
\end{align*}
hence $\widetilde D_2\cdot \delta=4\cdot (-\deg(q_!\cH))=4$, which is an even number. We conclude that $D_2$ is a Cartier divisor.
\end{proof}

\begin{rem}\label{rem.class.D2*} Let $m$ be the multiplicity of $\Sigma_2$ in the Cartier divisor $D_2$. Following the notation used in the proof of Proposition \ref{prop.D.Cartier}, we have $D_2^*=\widetilde D_2+m\widetilde\Sigma_2$, and we proved  $4=\widetilde D_2\cdot\delta=2m$. It follows $m=2$, i.e.\
\begin{equation}\label{eq.class.D*}
D_2^*=\widetilde D_2+2\widetilde \Sigma_2\textrm{\ \ in }H^2(\widetilde\M_{v_2},\ZZ).
\end{equation}
We conclude that the divisor $D_2^*$ is ruled by non-reduced curves, then we prefer to work with the divisor $\widetilde D_2$ instead (cf. Remark \ref{rem.pullback.strict.trans}). 
\end{rem}

Following Strategy \ref{strategy} we will compute the class of $D_2^*$ in $H^2(\widetilde\M_{v_2},\ZZ)$; we will deduce the class of $\widetilde D_2$ thanks to the expression  \eqref{eq.class.D*}.

\begin{prop}\label{prop.inters.D2} Let $D_1\subseteq\M_{v_2}$ be the divisor introduced in Definition \ref{def.D2}, and $\Gamma_2,\cT_2\subseteq\M_{v_2}$ the curves introduced in Definition \ref{def.curves2}. Then:
\begin{enumerate}
\item $D_2\cdot \Gamma_2=20$.
\item $D_2\cdot\cT_2=38$.
\end{enumerate}
\end{prop}
\begin{proof} The intersections in the statement can be computed proceeding as done for the divisor $D_1$ in Proposition \ref{prop.D.Gamma} and Proposition \ref{prop.D.T}. For this reason, we will underline the necessary modifications and then only sketch the computations here.
\begin{enumerate}
\item Once considered the Jacobian $J^6(\gamma)$ instead of the Jacobian $J^8(\gamma)$ in the proof of Proposition \ref{prop.D.Gamma}, the proof can be followed without further changes, and gives the same result.
\item Differently from the case of the divisor $D_1$, the curve $\cT_2$ is not contained in the stable locus $\M_{v_2}$: on the three singular curves $C_1\cup C'_1$, $C_2\cup C'_2$ and $C_3\cup C'_3$ parametrized by $\tau$, the corresponding sheaves in the curve $\cT_2$ are strictly semistable. Nevertheless, working on the locus of sheaves supported on smooth curves, we can deduce, proceeding as in the proof of Proposition \ref{prop.D.T} and using the same notation introduced there, that:
\[
(D_2\cdot\cT_2)|_U=\deg\Bigl(R^1p_{\rho\times\tau*}(p^*_{\mathscr C}\cE_{\cT_2}\otimes\cI^{\otimes 2}_{\hat R})\Bigl)|_U.
\]
Regarding the singular curves $C_1\cup C'_1$, $C_2\cup C'_2$ and $C_3\cup C'_3$, the situation here is slightly different from the one in Proposition \ref{prop.D.T}. When the sheaf $[F]\in\cT_2$, supported on a singular curve, has sections once removed a double point in $\rho_0$ from the component of its support containing the base point $q_1$, then $[F]\in D_2$ (see the construction of the limit curve in the proof of Lemma \ref{lemma.D2containsSigma} and of Proposition \ref{prop.D.Cartier} for a similar argument); when the point in $\rho_0$ lies on the component non containing $q_1$ then $[F]\notin D_2$ (by generality of the pencil $\tau$, with an argument analogous to the one given in the proof of Proposition \ref{prop.D.T}). Let $C_1$, $C_2$ and $C_3$ be the irreducible components containing the base point $q_1$, and call:
\[
\{x_i,y_i\}:=C_i'\cap\rho_0,\ \ i=1,2,3
\]
\[
B_i:=\{x_i\}\times\widetilde C_i+\{y_i\}\times\widetilde C_i\in \Pic(X).
\]
Then, if we define:
\[
L:=\cO_X(3\hat E_1+\hat E_2+\hat E_3+\hat E_4+B_1+B_2+B_2-2\hat R)
\]
from the discussion above and arguing as in the proof of Proposition \ref{prop.D.T}, we conclude:
\[
D_2\cdot\cT_2=\deg(R^1\hat q_* L)=-c_1(\hat q_!L).
\]
We proceed as in the proof of Proposition \ref{prop.D.T}. By the Grothendieck-Riemann-Roch theorem on the smooth variety $X$, one has: 
\[
\ch(\hat q_!L)=\hat q_*\bigl(\ch(L)\td(X)\bigl)\td(\hat\rho)^{-1}
\]
where:
\begin{itemize}
\item $\td(\hat\rho)^{-1}=(1,-1)$
\item $\td(X)=\bigl(1,-3\xi-\frac{1}{2}\sum_{i=1}^8[\hat E_i],-8\bigl)$
\item $\ch(L)=(1,[3\hat E_1+\hat E_2+\hat E_3+\hat E_4]+[B_1+B_2+B_3]-2[\hat R],-18)$.
\end{itemize}
It follows that:
\begin{align*}
\ch(L)\td(X)&=(1,\ch_1(L)+\td_1(X),-8+\ch_1(L)\td_1(X)-18) \\
&=(1,\ch_1(L)+\td_1(X),-38)
\end{align*}
then $\hat q_*\bigl(\ch(L)\td(X)\bigl)=(0,-38)$ and $D_2\cdot \cT=38$.
\end{enumerate}
\end{proof}

We can finally prove Theorem \ref{thm.q(tildeD2)}. 

\begin{proof}[Proof of Theorem \ref{thm.q(tildeD2)}]\label{proof.q(tildeD2)}
The divisor \(D_2\) is a linear combination of the \(\ZZ\)-generators \(\lambda_{v_2}(e)\) and \(\lambda_{v_2}(f)\) of \(\Pic(\M_{v_2})\), where \(e,f\in (v_2^\perp)^{1,1}\) have been defined in \eqref{eq.e.f.in.M2}. Combining the intersections computed in Proposition \ref{prop.inters.lambda4} and in Proposition \ref{prop.inters.D2}, we obtain that:
\begin{equation}\label{eq.class.D2}
D_2=-2\lambda_{v_2}(e)+22\lambda_{v_2}(f).
\end{equation}
Combining \eqref{eq.class.D2} and \eqref{eq.class.D*}, we obtain that \(\widetilde D_2\) and \(B_2\) correspond to the following classes in \(\Gamma_{v_2}\):
\[
f_{v_2}^{-1}(\widetilde D_2)=\bigl((-4,-2h,22),-2\sigma \bigl)\in \Gamma_{v_2}
\]
\[
f_{v_2}^{-1}(B_2)=\bigl((-4,-2h,22),-\sigma \bigl)\in \Gamma_{v_2}.
\]
By Theorem \ref{thm.Gamma_v} we conclude:
\[
q_{10}(\widetilde D_2)=\bigl((-4,-2h,22),-2\sigma \bigl)^2=(-4,-2h,22)^2+4\sigma^2=160
\]
\[
q_{10}(B_2)=\bigl((-4,-2h,22),-\sigma \bigl)^2=(-4,-2h,22)^2+\sigma^2=178.
\]
\end{proof}


\section{Ample uniruled divisors on $\OG10$ manifolds}\label{section:ample.uniruled.divisors}

In this section we will present our conclusions about the existence of ample uniruled divisors on $\OG10$-manifolds. 

\subsection{Deformation of rational curves}\label{subsection:defo.rat.curves}

The results we will present in this section are a slight modification of results proved by Charles, Mongardi and Pacienza in \cite[Section 3]{CPM}. We start by introducing some notation.

Let $p:\mathcal X\rightarrow B$ be a smooth projective morphism among quasi-projective varieties of relative dimension $2n$, and let $\alpha\in \Gamma(R^{4n-2}p_*\ZZ, B)$ be a class of type $(2n-1,2n-1)$. Under these hypothesis one can consider the relative Kontsevich moduli stack $\overline{\mathcal M_0}(\mathcal X/B,\alpha)$ of genus zero stable curves, whose points parametrize maps $f:C\rightarrow X$ with $C$ a stable curve of genus 0 and $X=\mathcal X_b$ a fiber of $\pi$, such that $f_*[C]=\alpha_b$; we will denote such a point by $[f]$. Note that the natural map $\overline{\mathcal M_0}(\mathcal X/B,\alpha)\rightarrow B$ is proper.

Let now $X$ be a projective IHS manifold of dimension $2n$ and $f:C\rightarrow X$ a fixed map from a stable genus 0 curve $C$; we also assume that $f$ is unramified along the generic point of any irreducible component of $C$. Let $\mathcal X\rightarrow B$ as above, with central fiber $\mathcal X_0=X$, $0\in B$; let $\alpha\in \Gamma(R^{4n-2}p_*\ZZ, B)$ as above, with $\alpha_0=f_*[C]$ in $H^{4n-2}(X,\ZZ)$.

Under this notation, we state the following results, both contained in \cite{CPM}. Our goal is to arrive to Corollary \ref{cor.defo.curves}, whose proof is a consequence of these two results.

\begin{prop}[Proposition 3.1 in \cite{CPM}]\label{prop.defo} In the setting introduced above, let \(X\) be a projective manifold with trivial canonical bundle and let $M\subset \overline{M_0}(X,f_*[C])$ be an irreducible component containing the point $[f]$. Then $\dim M\ge 2n-2$, and if $\dim M=2n-2$ then any irreducible component $\mathcal M\subset \overline{\mathcal M_0}(\mathcal X/B,\alpha)$ containing $[f]$ dominates $B$.
\end{prop}

\begin{prop}[Proposition 3.2 in \cite{CPM}]\label{prop.rat.equiv} Let $X$ be a projective $2n$-dimensional manifold endowed with a symplectic form and let $Y\subset X$ be subvariety of codimension $k$. If $W\subset X$ is a subvariety such that any point of $Y$ is rationally equivalent to a point in $W$, then the codimension of $W$ in $X$ is at most $2k$.
\end{prop}

From these results, one can conclude that it is possible to deform rational curves ruling a divisor along the base $B$, and that the deformations still rule a divisor. Charles, Mongardi and Pacienza proved this result in the case of irreducible and reduced rational curves (see Corollary 3.5 of \cite{CPM}, whose statement is analogous to the one of Corollary \ref{cor.defo.curves} here below), but in what follows we will need it also for reducible curves. For this reason we will prove here the reducible case, even if its proof is only a slight modification of the one given in \cite{CPM}.

\begin{cor}\label{cor.defo.curves} Let $f:C\rightarrow X$ be a non constant map from a possibly reducible stable genus zero curve $C$, and let $M$ be an irreducible component of $\overline{M_0}(X,f_*[C])$ containing $[f]$. Let $D\subset X$ be the subscheme covered by the deformations of $f$ parametrized by $M$. If $D$ is a divisor, then
any irreducible component of $\overline{\mathcal M_0}(\mathcal X/B,\alpha)$ containing $[f]$ dominates $B$. Furthermore, $\mathcal X_b$ contains a uniruled divisor $D_b$ for any point $b\in B$, whose cohomology class is a multiple of the cohomology class of $D$.
\end{cor}

\begin{proof}
By Proposition \ref{prop.defo}, $\dim M\ge 2n-2$; we want to prove that the equality holds, so that Proposition \ref{prop.defo} will imply that under our hypothesis
any irreducible component of $\overline{\mathcal M_0}(\mathcal X/B,\alpha)$ containing $[f]$ dominates $B$. 

Let us assume $\dim M\ge 2n-1$, and let $\mathscr C\rightarrow M$ the universal curve. Let us consider the evaluation map $\mathscr C\rightarrow D\subset X$, and let $D_0\subset D$ an irreducible component of $D$ ruled by the deformations of $f|_{C_0}:C_0\rightarrow X$, with $C_0$ rational irreducible component of $C$. Since we are assuming $\dim M\ge 2n-1$, the fibers of the evaluation map on $D_0$ have dimension at least 1, which means that there exists a subvariety $W_0\subset X$ with $\dim W_0\le \dim D-2=2n-3$ such that any point of $D_0$ is rationally equivalent to a point of $W_0$. Choosing $Y=D_0$ and $W=W_0$ in Proposition \ref{prop.rat.equiv}, this gives a contradiction.

The last part of the statement follows as in the proof of Corollary 3.5 in \cite{CPM}; for completeness, we repeat here their argument. Let $\mathcal M\subset \overline{\mathcal M_0}(\mathcal X/B,\alpha)$ be an irreducible component containing $M$, and let $\mathcal D\subset \mathcal X\rightarrow B$ be the locus covered by deformations of $f$ parametrized by $\mathcal M$. Any irreducible component of $\mathcal D$ dominates $B$, because from what we said before $\mathcal M$ dominates $B$. The central fiber of $\mathcal D\rightarrow B$ is by construction $D$, which is a divisor in $\mathcal X_0$; as consequence, the fiber of $\mathcal D\rightarrow B$ is a divisor $D_b\subset \mathcal X_b$ at any point $b\in B$, which is uniruled by construction.
\end{proof}

\subsection{Monodromy invariants}\label{subsection:monodromy.invariants}

We denote by \(\fM_{\OG10}^{pol}\) the moduli space of polarized irreducible holomorphic symplectic manifolds of \(\OG10\) type; we refer to \cite[Section 8]{markman.survey} for the construction of this moduli space. Connected components of the moduli space \(\fM_{\OG10}^{pol}\) coincide with the orbits under the action of the polarized monodromy group of the manifold. By \cite[Theorem 5.4]{onorati.monodromy}, the monodromy group \(\Mon^2(X)\) of the \(\OG10\) manifold \(X\) is maximal, i.e.\ \(\Mon^2(X)\) coincides with the group of orientation preserving isometries \(\Or^+(H^2(X,\ZZ))\) (cf. \cite{markman.survey} for the definition). Since the lattice \((H^2(X,\ZZ),q_X)\) contains two copies of the hyperbolic lattice \(U\) (\cite[Theorem 3.0.11]{rapagnettaOG10}), we can apply Eichler's criterion (see \cite[Section 10]{eichler} or \cite[Proposition 3.3,(i)]{gritsenko.hulek.sank}), getting that the orbit of an element \(\alpha\in H^2(X,\ZZ)\) via the action of the polarized monodromy group is determined by the degree and the divisibility of \(\alpha\) in the lattice \((H^2(X,\ZZ),q_X)\).

We conclude that the connected components of the moduli space \(\mathfrak M_{\OG10}^{pol}\) are characterized by the degree and the divisibility of the polarization, meaning their degree and divisibility in the second cohomology lattice of the manifold, endowed with the Beauville-Bogomolov-Fujiki form. In what follows, we will denote by \(\mathfrak M_{(d,l)}\) the irreducible component of \(\mathfrak M_{\OG10}^{pol}\) given by degree \(d\) and divisibility \(l\). 

We have already computed the degrees of the divisors \(\widetilde D_1\), \(B_1\), \(\widetilde D_2\) and \(B_2\); we compute here their divisibilities.

\begin{lem}\label{lem.divisibility} Consider the uniruled divisors $\widetilde D_1,B_1\subseteq\widetilde\M_{v_1}$ and and $\widetilde D_2, B_2\subseteq\widetilde\M_{v_2}$ introduced in Section \ref{section:class of the divisors}. Then:
\begin{enumerate}
\item $\divv(\widetilde D_1)=4$
\item $\divv(B_1)=2$
\item $\divv(\widetilde D_2)=4$
\item $\divv(B_2)=2$.
\end{enumerate}
\end{lem}
\begin{proof} This is a straightforward consequence of the expression of  \(f_{v_1}^{-1}\widetilde D_1\), \(f_{v_1}^{-1}B_1\), \(f_{v_2}^{-1}\widetilde D_2\) and \(f_{v_2}^{-1}B_2\) computed in Corollary \ref{cor.D1*.ruled.pos} and Theorem \ref{thm.q(tildeD2)}. Indeed, by Theorem \ref{thm.Gamma_v} the maps \(f_{v_1}\) and \(f_{v_2}\) are isometries, hence it is enough to compute the divisibility of the classes \(f_{v_1}^{-1}\widetilde D_1\), \(f_{v_1}^{-1}B_1\), \(f_{v_2}^{-1}\widetilde D_2\) and \(f_{v_2}^{-1}B_2\) in the lattices \(\Gamma_{v_1}\) and \(\Gamma_{v_2}\) respectively, which are found to be as in the statement by an easy and straightforward computation in the lattices \(\Gamma_{v_1}\) and \(\Gamma_{v_2}\). 
\end{proof}

\begin{rem} Observe that the classes \(B_1\) and \(B_2\) are primitive while \(\widetilde D_1\) and \(\widetilde D_2\) are four times a primitive class, as it is straightforward computed in the lattices \(\Gamma_{v_1}\) and \(\Gamma_{v_2}\) for the preimages of these classes via \(f_{v_1}\) and \(f_{v_2}\).
\end{rem}

We can finally state the main result of this article.

\begin{thm}\label{thm.uniruled.conn.comp} For any polarized irreducible holomorphic symplectic manifold \\ $(X,c_1(D))\in\mathfrak M_{(28,1)}\cup \mathfrak M_{(10,1)}\cup \mathfrak M_{(178,2)}\cup \mathfrak M_{(442,2)}\subseteq\mathfrak M_{\OG10}^{pol}$ there exists a positive integer \(m\) such that the linear system \(|mD|\) contains an element that is a uniruled divisor. 
\end{thm}
\begin{proof} 
By Theorem \ref{thm.q(D1*)}, Corollary \ref{cor.D1*.ruled.pos} and Theorem \ref{thm.q(tildeD2)} we know that the divisors \(\widetilde D_1\), \(B_1\), \(\widetilde D_2\) and \(B_2\) are positive effective divisors on the IHS manifolds \(\widetilde\M_{v_1}\) and \(\widetilde\M_{v_2}\) respectively. As consequence of Remark \ref{rem.need.positivity}, there exist small deformations \((X_1,H_1)\) of \((\widetilde\M_{v_1},\widetilde D_1)\),  \((X_2,H_2)\) of \((\widetilde\M_{v_1},B_1)\), \((X_3,H_3)\) of \((\widetilde\M_{v_2},\widetilde D_2)\) and \((X_4,H_4)\) of \((\widetilde\M_{v_2},B_2)\) such that \(H_i\) is an ample divisor for any \(i=1,2,3,4\). Since the divisors \(H_i\) have been obtained as deformation of the divisors \(\widetilde D_j\) and \(B_j\), they have the same monodromy invariants: by Corollary \ref{cor.D1*.ruled.pos}, Theorem \ref{thm.q(tildeD2)} and Lemma \ref{lem.divisibility} we have:
\begin{enumerate}
\item \(q_{10}(H_1)=448\), \(\divv(H_1)=4\)
\item \(q_{10}(H_2)=442\), \(\divv(H_2)=2\)
\item \(q_{10}(H_3)=160\), \(\divv(H_3)=4\)
\item \(q_{10}(H_4)=178\), \(\divv(H_1)=2\).
\end{enumerate}

Take \((X,c_1(D))\) as in the statement. When \((X,c_1(D))\in\mathfrak M_{(28,1)}\cup \mathfrak M_{(10,1)}\), the pair \((X,4c_1(D))\) belongs to the same connected component of \((X_i,c_1(H_i))\) for \(i=1\) or 3, since \(4c_1(D)\) and \(H_i\) have the same degree and divisibility. For the same reason, when \((X,c_1(D))\in \mathfrak M_{(178,2)}\cup \mathfrak M_{(442,2)}\) then the pair \((X,c_1(D))\) belongs to the same connected component of \((X_i,c_1(H_i))\) for \(i=2\) or 4. Assume to be in the case \(i=1\) or 3. It follows that the pair \((X,4D)\) is deformation of the pair \((\widetilde\M_{v_j},\widetilde D_j)\) for some \(j=1\) or 2. The divisor \(\widetilde D_j\) is ruled by (irreducible and) reduced curves, then by Corollary \ref{cor.defo.curves} there exists a positive integer \(m'\) such that the linear system \(|4m'D|\) contains an element which is uniruled. We conclude the statement choosing \(m:=4m'\). When the pair \((X,c_1(D))\) belongs to the connected component of \((X_i,c_1(H_i))\) for \(i=2\) or 4, we can proceed in the same way as before with the polarized pair \((X,c_1(D))\) instead of \((X,4c_1(D))\), since the rational curves ruling the divisors \(B_j\) are reducible but reduced for \(j=1,2\), hence we are in the hypotheses of Corollary \ref{cor.defo.curves}.
\end{proof}

\subsection{Developments}
To conclude, few words about possible developments of our work. Theorem \ref{thm.uniruled.conn.comp} proves the existence of ample uniruled divisors inside four connected components of $\fM_{\OG10}^{pol}$; it is very natural to wonder if it is possible to adapt our construction of uniruled divisors to cover more connected components of $\fM_{\OG10}^{pol}$. We want to remark that each connected component of \(\fM_{\OG10}^{pol}\) contains an element of the form \((\widetilde\M_v(S,H), c_1(D))\), for some projective \(\K3\) surface and some non primitive Mukai vector \(v\), with \(D\) ample divisor on \(\widetilde\M_v(S,H)\): this is again consequence of the maximality of the monodromy group, since on the moduli spaces \(\widetilde \M_v\) it is possible to construct polarizations of each possible degree and divisibility. For this reason, we do not exclude that it is possible to generalize our definition of uniruled divisors, obtaining the existence of ample uniruled divisors in many (possibly infinitely many) connected components of the moduli space $\fM_{\OG10}^{pol}$.


\end{document}